\documentclass[11pt,twoside,reqno]{amsart}
\usepackage{microtype}
\usepackage[OT1]{fontenc}
\usepackage{type1cm}
\usepackage{amssymb}
\usepackage[dvips]{graphicx}
\usepackage[bf]{caption}
\usepackage{psfrag}          
\usepackage[left=2.3cm,top=3cm,right=2.3cm]{geometry}        
\usepackage{version}         
\usepackage[11pt]{moresize}
\usepackage{subfigure}
\usepackage[usenames,dvipsnames]{xcolor}
\usepackage[colorlinks = true, citecolor = black, linkcolor = black, urlcolor = black, pdfstartview=FitH]{hyperref}

\numberwithin{equation}{section}

\theoremstyle{plain}
\newtheorem{theorem}{Theorem}[section]

\newtheorem{proposition}[theorem]{Proposition}
\newtheorem{lemma}[theorem]{Lemma}

\theoremstyle{remark}
\newtheorem{remark}[theorem]{Remark}

\newtheorem{example}[theorem]{Example}

\theoremstyle{definition}
\newtheorem{definition}[theorem]{Definition}

\newtheorem*{notation*}{Notation}

\renewcommand{\AA}{\mathcal{A}}
\newcommand{\HH}{\mathcal{H}}
\newcommand{\LL}{\mathcal{L}}
\newcommand{\PP}{\mathcal{P}}
\newcommand{\MM}{\mathcal{M}}

\newcommand{\R}{\mathbb{R}}

\newcommand{\Q}{\mathbb{Q}}
\newcommand{\N}{\mathbb{N}}

\newcommand{\cB}{\mathcal{B}}
\newcommand{\cV}{\mathcal{V}}

\newcommand{\cM}{\mathcal{M}}
\newcommand{\cH}{\mathcal{H}}
\newcommand{\cL}{\mathcal{L}}

\newcommand{\cU}{\mathcal{U}}

\newcommand{\cP}{\mathcal{P}}

\newcommand{\cA}{\mathcal{A}}

\newcommand{\TD}{\mathcal{TD}}
\newcommand{\FD}{\mathcal{FD}}
\newcommand{\EFD}{\mathcal{EFD}}

\newcommand{\eps}{\varepsilon}

\newcommand{\roo}{\varrho}

\newcommand{\dist}{\operatorname{dist}}

\newcommand{\Tan}{\operatorname{Tan}}

\newcommand{\Hd}{\dim_\mathrm{H}}

\renewcommand{\epsilon}{\varepsilon}
\renewcommand{\rho}{\varrho}
\renewcommand{\phi}{\varphi}

\newcommand{\ol}{\overline}

 \newcommand{\yli}[2]{\genfrac{}{}{0pt}{}{#1}{#2}}

  \newcommand{\la}{\langle}
  \newcommand{\ra}{\rangle}

\DeclareMathOperator{\proj}{proj}

\DeclareMathOperator{\spt}{spt}

\DeclareMathOperator{\dimloc}{dim_{loc}}
\DeclareMathOperator{\udimloc}{\overline{dim}_{loc}}
\DeclareMathOperator{\ldimloc}{\underline{dim}_{loc}}

\DeclareMathOperator{\ulocd}{\overline{dim}_{loc}}
\DeclareMathOperator{\llocd}{\underline{dim}_{loc}}

\DeclareMathOperator{\dimh}{dim_H}
\DeclareMathOperator{\udimh}{\overline{dim}_H}
\DeclareMathOperator{\ldimh}{\underline{dim}_H}

\DeclareMathOperator{\dimp}{dim_p}
\DeclareMathOperator{\udimp}{\overline{dim}_p}
\DeclareMathOperator{\ldimp}{\underline{dim}_p}

\DeclareMathOperator*{\esssup}{ess\,sup}
\DeclareMathOperator*{\essinf}{ess\,inf}

\begin{document}

\title[Scenery flow and geometry of measures]{Dynamics of the scenery flow and geometry of measures}

 \author{Antti K\"aenm\"aki}
 \address{Department of Mathematics and Statistics \\
          P.O.\ Box 35 (MaD) \\
          FI-40014 University of Jyv\"askyl\"a \\
          Finland}
 \email{antti.kaenmaki@jyu.fi}

 \author{Tuomas Sahlsten}
 \address{Einstein Institute of Mathematics\\
	The Hebrew University of Jerusalem\\
         Givat Ram, Jerusalem 91904 \\
         Israel}
\email{tuomas@sahlsten.org}

 \author{Pablo Shmerkin}
 \address{Department of Mathematics and Statistics\\
 	Torcuato Di Tella University\\
	Av. Figueroa Alcorta 7350, Buenos Aires\\ 
	Argentina}
\email{pshmerkin@utdt.edu}

\thanks{T.S. acknowledges the partial support from the University of Bristol, the Finnish Centre of Excellence in Analysis and Dynamics Research, the Emil Aaltonen Foundation and European Union (ERC grant $\sharp$306494). P.S. was partially supported by a Leverhulme Early Career Fellowship and by Project PICT 2011-0436 (ANPCyT)}
\dedicatory{Dedicated to Professor Pertti Mattila on the occasion of his 65th birthday}

\maketitle

\begin{abstract}
We employ the ergodic theoretic machinery of scenery flows to address classical geometric measure theoretic problems on Euclidean spaces. Our main results include a sharp version of the conical density theorem, which we show to be closely linked to rectifiability. Moreover, we show that the dimension theory of measure-theoretical porosity can be reduced back to its set-theoretic version, that Hausdorff and packing dimensions yield the same maximal dimension for porous and even mean porous measures, and that extremal measures exist and can be chosen to satisfy a generalized notion of self-similarity. These are sharp general formulations of phenomena that had been earlier found to hold in a number of special cases.
\end{abstract}


\section{Introduction}

Ergodic theory studies the asymptotic behaviour of typical orbits of dynamical systems endowed with an invariant measure. Geometric measure theory can be described as a field of mathematics where geometric problems on sets and measures are studied via measure-theoretic techniques. Although a priori it may seem that these subjects are disconnected, recently some deep links between them have been uncovered. The main idea is to study the structure of a measure $\mu$ on $\R^d$ via dynamical properties of its magnifications at a given point $x \in \R^d$. The resulting family of measures, i.e.\ the \textit{scenery} $(\mu_{x,t})_{t \geq 0}$ defined by
\[
\mu_{x,t}(A)  = \frac{\mu(e^{-t}A + x)}{\mu(\overline{B}(x,e^{-t}))}, \quad A \subset \R^d,
\]
where $\overline{B}(x,e^{-t})$ is the closed ball of center $x$ and radius $e^{-t}$, can be interpreted as an orbit in a dynamical system, whose evolution is described a measure valued flow known as the \textit{scenery flow}. We emphasize that the measures $\mu_{x,t}$ are restricted to the unit ball and normalized to be probability measures. While the idea behind the scenery flow is far from new, until now authors had either considered the scenery flow for specific sets and measures arising from dynamics (see e.g.\ \cite{Zahle88,BedfordFisher1996,BedfordFisher1997}), investigated abstract scenery flows but with a view on applications to special sets and measures, again arising from dynamics or arithmetic (see e.g.\ \cite{Hochman2010,Gavish2011,HochmanShmerkin2014}), or established properties of the scenery flow for its own sake (see \cite{Morters1998,MortersPreiss1998}). The main innovation of this article is to employ the general theory initiated by Furstenberg \cite{Furstenberg2008}, greatly developed by Hochman \cite{Hochman2010}, and extended by the authors \cite{KaenmakiSahlstenShmerkin2014} with a view on our applications here, to classical problems in geometric measure theory which a priori do not involve any dynamics. The power of the theory around the scenery flow allow us to to obtain very sharp versions of existing results, simplify the proofs of others, and prove in great generality certain phenomena that had been previously observed only in special cases. However, rather than individual results, we believe that our main contribution is to highlight the relevance of ergodic-theoretic methods around the scenery flow in geometric problems. We hope this approach will find further applications in geometric measure theory and analysis.

One of the oldest and most fundamental concepts in analysis is that of \textit{tangent}. Tangents capture the local structure of functions but are substantially better behaved. In particular, they have scaling and isotropy properties the original object lacks. One can infer global properties of the object from the collection of tangents at all points. The analogous concept for general measures is that of \textit{tangent measure} introduced by Preiss \cite{Preiss1987}. These measures are the accumulation points of the scenery $\mu_{x,t}$ as $t \to \infty$. Although tangent measures share many of the good properties of tangents, tangent measures do not give much information about the original object: two very different measures may have the same set of tangent measures at every point. It may even happen for a measure to have \emph{all} measures as tangent measures, at almost every point; see O'Neil \cite{ONeil1995} and Sahlsten \cite{Sahlsten2015}. In order to obtain a notion of tangent to a measure that captures more information, a simple but deep idea is to look at the \emph{statistics} of the scenery $(\mu_{x,t})_{t \geq 0}$, rather than the collection $\Tan(\mu,x)$ of all the accumulation points. This leads to the notion of \textit{tangent distribution}, studied (with some variations and under different names) by several authors, see e.g. \cite{Zahle88,MortersPreiss1998,Hochman2010,Gavish2011,KaenmakiSahlstenShmerkin2014}. Tangent distributions are defined to be weak accumulation points of $\la \mu \ra_{x,T}$, $T \geq 0$, where
\[
\la \mu \ra_{x,T} = \frac{1}{T}\int_0^T \delta_{\mu_{x,t}} \, \mathrm{d}t.
\]
That is, tangent distributions are \emph{measures on measures} (or random measures), and from the dynamical point of view they are empirical measures for the scenery flow. The family of all tangent distributions of $\mu$ at $x$ is denoted by $\TD(\mu,x)$. Tangent distributions describe the asymptotics of the scenery flow and their support is a (possibly much smaller) subset of the set $\Tan(\mu,x)$ of all tangent measures. We will see that tangent distributions capture many properties of the original object that are not invisible to averaging, such as dimension (though it should be noted that tangent distributions are blind to properties that are sensitive to changes on a sparse set of scales, such as classical rectifiability).

The approach of this paper is to study properties of measures through the corresponding properties of their tangent distributions. Our goal is to apply ergodic theoretical methods and hence we restrict ourselves to \textit{scale invariant} tangent distributions. If we additionally require that the distribution satisfies a suitable isotropy condition called \textit{quasi-Palm} (see Definition \ref{def:FD}), then such a distribution is called a \textit{fractal distribution}. Roughly speaking, the quasi-Palm property guarantees that the null sets of the distribution are invariant under translation to a typical point for the measure. This kind of property turns out to be extremely useful when transferring information about a generic measure for the distribution from a \emph{single} point to \emph{almost every} point. A remarkable result of Hochman \cite[Theorem 1.7]{Hochman2010} states that tangent distributions are indeed fractal distributions at almost every point (while the scale invariance is perhaps expected, the fact that tangent distributions satisfy the quasi-Palm property is one of the main discoveries of \cite{Hochman2010}).

The geometric properties of fractals and their relation to different notions of dimension and rectifiability have been an object of intensive study for several decades, both for their intrinsic interest and because of connections to other areas. In particular, questions dealing with \textit{conical densities} and \textit{porosity} have been studied thoroughly. The theory that started with conical density properties of Hausdorff measures was later extended to more general Borel measures, which is the setting we consider. The study of porosity is motivated in part by the fact that it is a quantitative notion of singularity. Thus it is a natural problem to understand the relationship between dimensions (which are a different way to quantify singularity) and various notions of porosity. Particular interest has been given to the study of the maximal possible dimension of a porous set or measure. This problem has received considerable attention and it has been addressed by using various methods and techniques.

In this article, we show that, equipped with the machinery of fractal distributions, most of the above alluded results on conical densities and porosities are consequences of a suitable notion of rectifiability, and of set-theoretical porosity. In this way, we unify and explain in a coherent way a large number of previously piecemeal results. We obtain an essentially sharp and very general conical density result, and prove new results for different kinds of porosity valid for arbitrary values of the ``size of the hole'' (nearly all previous research in this direction focused on ``small'' or ``large'' holes only). Moreover, we prove that in these cases extremal measures exist and can be taken to be uniformly scaling, which is an ergodic-theoretic notion of self-similarity.

The article is organized as follows: In \S \ref{sec:results}, we give more detailed background on the geometric problems we study and state the main results. The ergodic theoretic machinery on scenery flows is presented in \S \ref{sec:fractaldistributions}, and \S \ref{sec:proofs} contains all the proofs.

\section{Background and statement of main results} \label{sec:results}

Let $\R^d$ be the $d$-dimensional Euclidean space equipped with the usual Euclidean metric. The open ball centered at $x$ with radius $r>0$ is denoted by $B(x,r)$. Moreover, let $\overline{B}(x,r)$ be the corresponding closed ball. For the closed unit ball, we write $B_1 := \overline{B}(0,1)$.

Given a metric space $X$, we denote the family of all Borel probability measures on $X$ by $\PP(X)$. When $X = B_1$, just write $\cM_1 = \cP(B_1)$. When $X$ is locally compact, $\PP(X)$ is endowed with the weak$^*$ topology (as usual we speak of weak convergence rather than weak$^*$ convergence). Let $\cM$ be the set of all Radon measures on $\R^d$ and $\spt \mu$ the topological support of $\mu\in \cM$ in $\R^d$. The $d$-dimensional Lebesgue measure is denoted by $\LL^d$ and in the case $d = 1$, we just write $\lambda = \LL^1$. We will also write $\overline{\LL}^d$ for the normalized restriction of $\LL^d$ to the unit ball $B_1$.  Moreover, $\cH^s$ is the $s$-dimensional Hausdorff measure defined by using the Euclidean metric on $\R^d$.

Following notation of Hochman \cite{Hochman2010}, we refer to elements of $\cM$ as \emph{measures}, and to elements of $\mathcal{P}(\cM_1)$ as \emph{distributions}. Measures will be denoted by lowercase Greek letters $\mu,\nu$, etc.\ and distributions by capital letters $P,Q$, etc. We use the notation $x \sim \mu$ if a point $x$ is chosen randomly according to a measure $\mu$. Moreover, write $\mu \sim\nu$ if the measures $\mu$ and $\nu$ are \textit{equivalent}, that is, they have the same null-sets. If $f$ is a function and $\mu$ is a measure, then $f\mu$ is the push-forward measure $A \mapsto \mu(f^{-1} A)$. Finally, if $\mu$ is a measure and $\mu(A) > 0$, we let $\mu_A:= \mu(A)^{-1}\mu|_A$ be the normalized restriction of $\mu$ on $A$.

The main idea that we pursue in this work is that many properties of a measure can be related to analog properties of their tangent distributions which, thanks to Hochman's result \cite[Theorem 1.7]{Hochman2010}, are much more structured objects and, as with any tangent objects, enjoy uniform versions of many geometric properties of the original measure. We are specially interested in the geometric notion of \textit{dimension} of sets and measures. We recall some definitions.

\begin{definition}[Local dimensions] If $\mu\in\MM$ and $x\in\R^d$, then the \emph{upper} and \emph{lower local dimensions} of $\mu$ at $x$ are defined by
\begin{align*}
\udimloc(\mu,x) = \limsup_{r\downarrow 0}\frac{\log\mu(B(x,r))}{\log r} \quad \text{and} \quad \ldimloc(\mu,x) = \liminf_{r\downarrow 0}\frac{\log\mu(B(x,r))}{\log r}.
\end{align*}
\end{definition}

If both values agree, then the common value is the \emph{local dimension} of $\mu$ at $x$, denoted by $\dimloc(\mu,x)$. A measure $\mu$ is \emph{exact-dimensional} if $\dimloc(\mu,x)$ exists and is $\mu$ almost everywhere constant. In this case, the common value is simply called the \emph{dimension} of $\mu$, denoted by $\dim\mu$. To cover situations where $\dim \mu$ does not exist, we use the notions of Hausdorff and packing dimensions.

\begin{definition}[Hausdorff and packing dimensions] If $\mu \in \MM$ then \emph{upper} and \emph{lower Hausdorff} and \textit{packing dimensions} of $\mu$ are defined by
\begin{align*}
  \ldimh \mu  &= \essinf_{x \sim \mu} \ldimloc(\mu,x), \\
  \udimh \mu &= \esssup_{x \sim \mu} \ldimloc(\mu,x), \\
  \ldimp \mu  &= \essinf_{x \sim \mu} \udimloc(\mu,x),\\
  \udimp \mu &= \esssup_{x \sim \mu} \udimloc(\mu,x).
\end{align*}
\end{definition}

Recall that these quantities can be recovered from the classical set-theoretical Hausdorff and packing dimensions as follows:
\begin{align*}
  \ldimh \mu &= \inf\{ \dimh A : A \subset \R^d \text{ is a Borel with } \mu(A)>0 \}, \\
  \udimh \mu &= \inf\{ \dimh A : A \subset \R^d \text{ is a Borel with } \mu(\R^d \setminus A)=0 \}, \\
  \ldimp \mu &= \inf\{ \dimp A : A \subset \R^d \text{ is a Borel with }\mu(A)>0 \}, \\
  \udimp \mu &= \inf\{ \dimp A : A \subset \R^d \text{ is a Borel with } \mu(\R^d \setminus A)=0 \}.
\end{align*}
Here on the right-hand side $\dimh$ and $\dimp$ denote Hausdorff and packing dimensions of sets. The reader is referred to the books of Mattila \cite{Mattila1995} and Falconer \cite{Falconer1997} for references and further background on measures and dimensions.

\subsection{Rectifiability and conical densities} \label{sec:conical}

Rectifiability is one of the most fundamental concepts of geometric measure theory. A \textit{rectifiable set} is a set that is smooth in a certain measure-theoretic sense. It is an extension of the idea of a rectifiable curve to higher dimensions. To a great extent, geometric measure theory is about studying rectifiable and purely unrectifiable sets. A set is \textit{purely unrectifiable} if its intersection with any rectifiable set is negligible. These concepts form a natural pair since every set can be decomposed into rectifiable and purely unrectifiable parts. Although a $k$-rectifiable set $E$ (with finite $\HH^k$ measure) bears little resemblance to smooth surfaces (for example, it can be topologically dense), it admits a measure-theoretical notion of tangent at all but a zero $\HH^k$-measure set of points; see for example \cite[Chapter 15]{Mattila1995}.

The foundations of geometric measure theory were laid by Besicovitch \cite{Besicovitch1928,Besicovitch1929,Besicovitch1938}. He introduced the theory of rectifiable sets by describing the structure of the subsets of the plane having finite $\HH^1$ measure. Besicovitch's work was extended to $k$-dimensional subsets of $\R^d$ by Federer \cite{Federer1947}. Morse and Randolph \cite{MorseRandolph1944}, Moore \cite{Moore1950}, Marstrand \cite{Marstrand1954,Marstrand1961,Marstrand1964}, and Mattila \cite{Mattila1975} studied extensively how densities are related to rectifiability. Preiss \cite{Preiss1987} managed to completely characterize $k$-rectifiable sets by the existence of $k$-dimensional density effectively by introducing and employing tangent measures. For various other characterizations and properties of rectifiability the reader is referred to the book of Mattila \cite{Mattila1995}.

Conical density results are used to derive geometric information from metric information. The idea is to study how a measure is distributed in small balls. Upper conical density results related to Hausdorff measure are naturally linked to rectifiability; see Besicovitch \cite{Besicovitch1938}, Marstrand \cite{Marstrand1954}, Federer \cite{Federer1969}, Salli \cite{Salli1985}, and Mattila \cite{Mattila1988,Mattila1995}. The works of K\"aenm\"aki and Suomala \cite{KaenmakiSuomala2008,KaenmakiSuomala2004}, Cs\"ornyei, K\"aenm\"aki, Rajala, and Suomala \cite{CsornyeiKaenmakiRajalaSuomala2010}, Feng, K\"aenm\"aki, and Suomala \cite{FengKaenmakiSuomala2012}, K\"aenm\"aki, Rajala, and Suomala \cite{KaenmakiRajalaSuomala2012}, and Sahlsten, Shmerkin, and Suomala \cite{SahlstenShmerkinSuomala2013} introduced conical density results for more general measures in more general settings.

Applications of conical densities have been found in the study of porosities; see Mattila \cite{Mattila1988} and K\"aenm\"aki and Suomala \cite{KaenmakiSuomala2008,KaenmakiSuomala2004}. They have also been applied in the removability questions for Lipschitz harmonic functions; see Mattila and Paramonov \cite{MattilaParamonov1995} and Lorent \cite{Lorent2003}.

It turns out that tangent distributions are well suited to address problems concerning conical densities. The cones in question do not change under magnification and this allows to pass information between the original measure and its tangent distributions. In fact, we will show that, perhaps surprisingly, most of the known conical density results are, in some sense, a manifestation of rectifiability.

Let $d \in \N$, $k \in \{ 0,\ldots,d-1 \}$, and $G(d,k)$ be the space of all $k$-dimensional linear subspaces of $\R^d$. The  unit sphere of $\R^d$ is $S^{d-1}$. For $x \in \R^d$, $V \in G(d,k)$, $\theta \in S^{d-1}$, and $0 \le \alpha \le 1$ we set
\begin{align*}
  X(x,r,V,\alpha) &= \{ y \in B(x,r) : \dist(y-x,V) < \alpha|y-x| \}, \\
  H(x,\theta,\alpha) &= \{ y \in \R^d : (y-x) \cdot \theta \geq \alpha|y-x| \}.
\end{align*}
Classical results of Besicovitch \cite{Besicovitch1938}, Marstrand \cite{Marstrand1954}, Salli \cite{Salli1985}, and Mattila \cite{Mattila1988} guarantee that if the Hausdorff dimension of the set is large enough, then there are arbitrary small scales so that almost all points of the set are effectively  surrounded by the set. Conical density results aim to give conditions on a measure (usually, a lower bound on some kind of dimension) which guarantee that the non-symmetric cones $X(x,r,V,\alpha) \setminus H(x,\theta,\alpha)$ contain a large portion of the mass from the surrounding ball $B(x,r)$, at many scales $r$ and at many points $x$.

\begin{figure}[t!]
\subfigure{\includegraphics[scale=0.45]{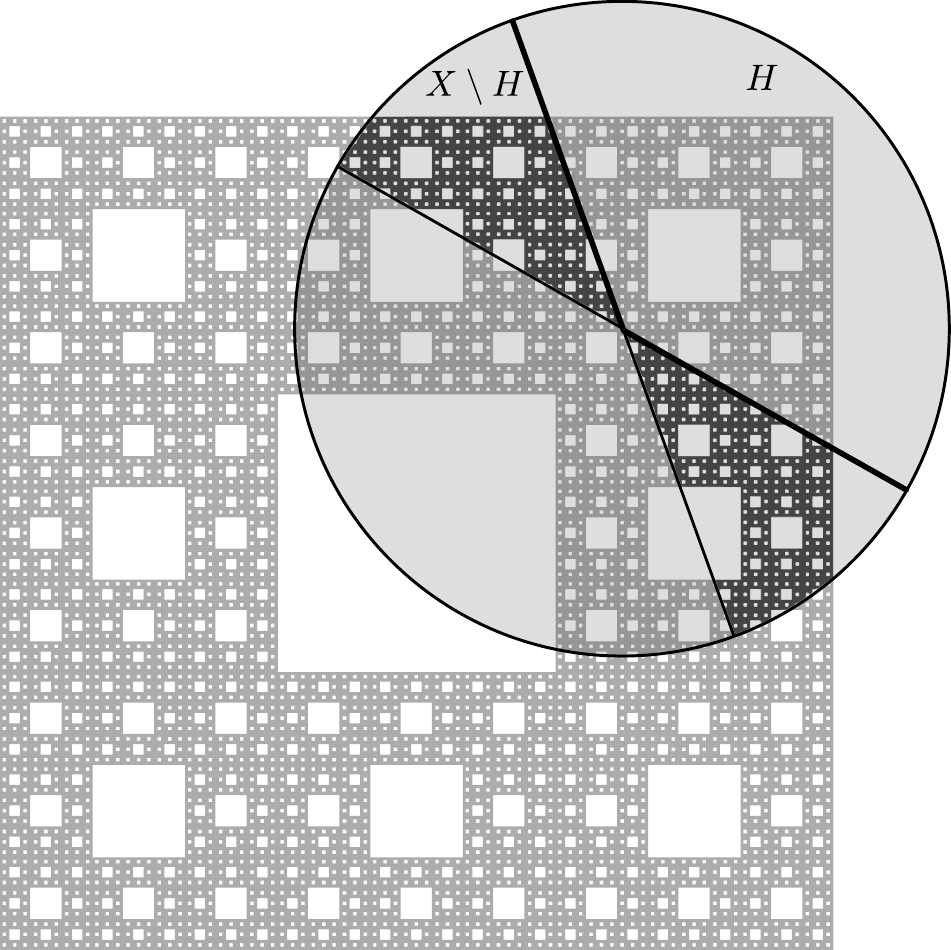}}
\quad
\subfigure{\includegraphics[scale=0.45]{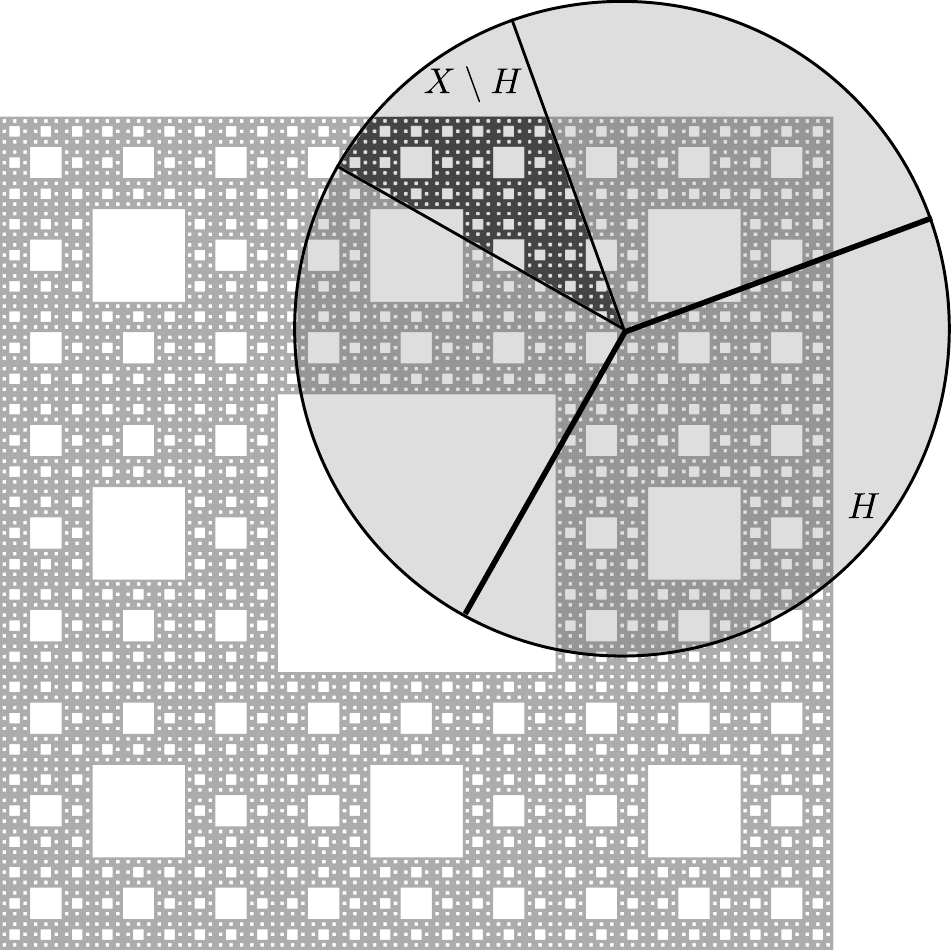}}
\quad
\subfigure{\includegraphics[scale=0.45]{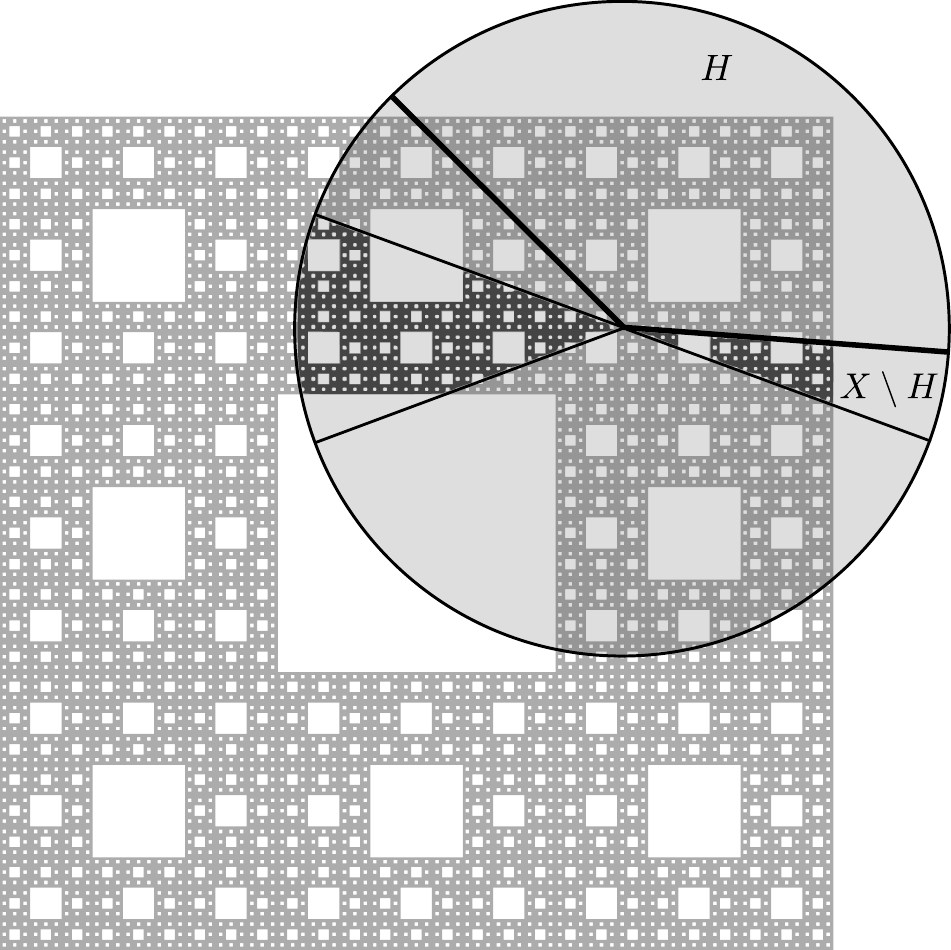}}
\caption{Conical density results quantify the scales that are ``spread-out'' by considering the proportion of relative mass in the cones $X(x,r,V,\alpha) \setminus H(x,\theta,\alpha)$ for all directions $V$ and $\theta$. In the picture, $X=X(x,r,V,\alpha)$ and $H=H(x,\theta,\alpha)$ when $k=1$ for a fixed small value of $\alpha$ and various directions $V$ and $\theta$.}
\end{figure}

In our first main result, we give an optimal quantitative estimate for the amount of the scales where such phenomenon occurs. For $d \in \N$, $k \in \{ 1,\ldots,d-1 \}$, and $0<\alpha\le 1$, we set
\begin{equation} \label{eq:def-eps}
  \eps(d,k,\alpha) := \inf  \left\{\frac{\LL^d(X(0,1,V,\alpha) \setminus H(0,\theta,\alpha))}{\LL^d(B(0,1))}: V\in G(d,d-k) \textrm{ and } \theta\in S^{d-1}\right\}.
\end{equation}
It follows from elementary geometry and the rotational invariance of Lebesgue measure that, in fact, the infimum is attained whenever $V\in G(d,d-k)$ and $\theta\in S^{d-1}\cap V$.


\begin{theorem} \label{thm:conical}
If $d \in \N$, $k \in \{ 1,\ldots,d-1 \}$, $k<s\le d$, and $0<\alpha\le 1$, then there exists $0 < \eps < \eps(d,k,\alpha)$ satisfying the following: For every Radon measure $\mu$ on $\R^d$ with $\ldimh \mu \ge s$ it holds that
\begin{equation}\label{eq:lowerconicalporous}
  \liminf_{T \to \infty} \frac{1}{T} \,\lambda\Big(\Big\{ t \in [0,T] : \inf_{\yli{\theta \in S^{d-1}}{V \in G(d,d-k)}} \frac{\mu(X(x,e^{-t},V,\alpha) \setminus H(x,\theta,\alpha))}{ \mu(\overline{B}(x,e^{-t}))} > \epsilon\Big\}\Big) \ge \frac{s-k}{d-k}
\end{equation}
at $\mu$ almost every $x \in \R^d$. If the measure $\mu$ only satisfies $\ldimp\mu \ge s$, then \eqref{eq:lowerconicalporous} holds with $\limsup_{T \to \infty}$ at $\mu$ almost every $x \in \R^d$.

Furthermore, this is sharp in the sense that for every $k<s<d$ there exists a Radon measure $\mu$ of exact dimension $s$ such that \eqref{eq:lowerconicalporous} holds with $\lim_{T \to \infty}$ and the limit equals $(s-k)/(d-k)$ for all $0<\eps<\eps(d,k,\alpha)$, but is equal to zero for all $\eps > \eps(d,k,\alpha)$.
\end{theorem}

\begin{remark}
(1) A similar result is available by Sahlsten, Shmerkin, and Suomala \cite[Theorem 1.2]{SahlstenShmerkinSuomala2013} but with some constant $p$ on the right-hand side instead of the sharp quantity $(s-k)/(d-k)$; although, in principle, the value of $p$ can be computed, it is clear that it is very far from the optimal value. In turn, \cite[Theorem 1.2]{SahlstenShmerkinSuomala2013} unified and extended most of the earlier results on conical densities; we refer to \cite{SahlstenShmerkinSuomala2013} for further discussion.
We underline that our method of proof is entirely different from, and in our view more conceptual than, that of \cite{SahlstenShmerkinSuomala2013} and other previous research on the topic.

(2) We have chosen the cones $X(x,r,V,\alpha)$ to be open and the cones $H(x,\theta,\alpha)$ to be closed in order to keep the proof simpler, but since these cones are nested as $\alpha$ decreases, Theorem \ref{thm:conical} holds regardless of whether the inequalities in their definitions are strict or not. Moreover, since for each $x$ there are at most countably many $t$ such that the boundary of $B(x,e^{-t})$ has positive $\mu$-mass, the result also holds if in the definition of $X(x,r,V,\alpha)$ we intersect with the closed ball $\overline{B}(x,r)$ instead. Likewise, in \eqref{eq:lowerconicalporous} we divide by the mass of the \emph{closed} ball $\overline{B}(x,e^{-t})$ in order to match the definition of the scenery flow but, since $\mu(B(x,e^{-t}))\le \mu(\overline{B}(x,e^{-t}))$, the result also holds if we divide by $\mu(B(x,e^{-t}))$ instead.  These observations also apply to Theorem \ref{thm:unrectifiable} below.
\end{remark}

The proof of Theorem \ref{thm:conical} is based on showing that there cannot be ``too many'' rectifiable tangent measures. Let us next give the precise definition for rectifiability and start examining its relationship to conical density results in more detail.

\begin{definition}[Rectifiability of sets and measures] \label{def:rect}
A set $E \subset \R^d$ is called \emph{$k$-rectifiable} if there are countably many Lipschitz maps $f_i \colon \R^k \to \R^d$ so that
$$\cH^k\Big(E \setminus \bigcup_i f_i(\R^k)\Big) = 0.$$
Moreover, we say that a Radon measure $\nu$ is \emph{$k$-rectifiable} if $\nu\ll \cH^k$ and there exists a $k$-rectifiable set $E\subset\R^d$ such that $\nu(\R^d\setminus E)=0$

A Radon measure $\mu$ is \emph{purely $k$-unrectifiable} if it gives no mass to $k$-rectifiable sets and $E$ is \emph{purely $k$-unrectifiable} if the restriction $\cH^k|_E$ is purely $k$-unrectifiable.
\end{definition}

While the definition above is nowadays standard and very useful in dealing with properties valid up to sets of zero $\cH^k$ measure (see e.g.\ \cite[Chapters 15--19]{Mattila1995}), we will also require a variant which corresponds to the definition of rectifiability in Federer \cite[3.2.14]{Federer1969}. To avoid any confusion, we call it \emph{strong} rectifiability.

\begin{definition}[Strong rectifiability of sets] \label{ded:strong-rect}
A set $E\subset\R^d$ is called \emph{strongly $k$-rectifiable} if there exist countably many Lipschitz maps $f_i \colon \R^k \to \R^d$ such that
$$E\subset \bigcup_i f_i(\R^k).$$
\end{definition}

\begin{remark}
A strongly $k$-rectifiable set is obviously $k$-rectifiable. On the other hand, there are many sets which are $k$-rectifiable but not strongly $k$-rectifiable, for example, any set $E\subset\R^d$ such that $\cH^k(E)=0$ but $\dimp E > k$.


\end{remark}

Pure unrectifiability is also a condition which should guarantee that the measure is scattered in many directions. Indeed, under suitable assumption, it leads to a conical density result: for example, the following is proved in K\"aenm\"aki \cite[Theorem 8]{Kaenmaki2010}.

\begin{theorem}
  If $M > 0$ and $0<\alpha\le 1$, then there is a constant $\eps>0$ depending only on $M$ and $\alpha$ satisfying the following: For every $d \in \N$, $k \in \{ 1,\ldots,d-1 \}$, $V \in G(d,d-k)$, and purely $k$-unrectifiable measure $\mu$ on $\R^d$ with
  \begin{equation} \label{eq:doubling}
    \limsup_{r \downarrow 0} \frac{\mu(\overline{B}(x,2r))}{\mu(\overline{B}(x,r))} < M
  \end{equation}
  at $\mu$ almost every $x \in \R^d$, it holds that
  \begin{equation*} 
    \limsup_{r \downarrow 0} \frac{\mu(X(x,r,V,\alpha))}{\mu(\overline{B}(x,r))} \ge \eps
  \end{equation*}
  at $\mu$ almost every $x \in \R^d$.
\end{theorem}

The doubling condition \eqref{eq:doubling} is a crucial assumption and, in fact, the result fails to hold for general purely unrectifiable measures. In Cs\"ornyei, K\"aenm\"aki, Rajala, and Suomala \cite[Example 5.5]{CsornyeiKaenmakiRajalaSuomala2010}, it was demonstrated that there exists $V \in G(2,1)$ and a purely $1$-unrectifiable measure $\mu$ on $\R^2$ such that for any $0 < \alpha < 1$ we have
$$\lim_{r \downarrow 0} \frac{\mu(X(x,r,V,\alpha))}{\mu(\overline{B}(x,r))} = 0$$
at $\mu$ almost every $x$. This reflects the fact that rectifiability can be broken by having the measure ``look unrectifiable'' at some very sparse sequence of scales. More precisely, a measure $\mu$ is $k$-rectifiable if and only if  at $\mu$ almost every point each tangent measure of $\mu$ is a constant times $\HH^k_{V\cap B_1}$ for some $V \in G(d,k)$; see Mattila \cite[Theorem 16.5]{Mattila1995}.
A particular consequence of this is that if $\mu$ is $k$-rectifiable then at $\mu$ almost every $x$, any $P \in \TD(\mu,x)$ satisfies
$$
  P(\{ \nu \in \cM_1 : \spt \nu \text{ is strongly $k$-rectifiable} \}) = 1.
$$
See Lemma \ref{lem:tangents-of-rectifiable} below for a more precise statement. It is important to remark that a purely $k$-unrectifiable measure $\mu$ can satisfy $P(\{\nu \in \cM_1 : \nu$ is not strongly $k$-rectifiable$\}) = 0$ for all $P \in \TD(\mu,x)$ at $\mu$ almost every $x$; consult e.g.\ the example of David and Semmes \cite[\S 20]{DavidSemmes1991}.

We will next introduce a quantitative notion of average unrectifiability. It describes the proportion of scales where we can see strongly unrectifiable sets, as measured by tangent distributions of $\mu$ at $\mu$ typical points:

\begin{definition}[Average unrectifiability] \label{def:average}
Given a proportion $0 \leq p < 1$, we say that a measure $\mu \in \cM$ is \emph{$p$-average $k$-unrectifiable} if we have
$$
  P(\{\nu \in \cM_1 : \spt\nu \text{ is not strongly $k$-rectifiable}\}) > p
$$
for every $P \in \TD(\mu,x)$ at $\mu$ almost every $x$.
\end{definition}

\begin{example}
\label{ex:unrect}
(1) Let $\mu$ be a self-similar measure supported on a self-similar set of dimension $k$ satisfying the strong separation condition. For example, let $\mu$ be the $1$-dimensional Hausdorff measure in $\R^2$ restricted to the product of two $\tfrac14$-Cantor sets. Then all tangent measures of $\mu$ at a $\mu$ typical point are restrictions of homothetic copies of the measure itself. Recall e.g.\ Bandt \cite{Bandt2001}. Since self-similar sets of dimension $k$ are purely unrectifiable and in particular are not strongly rectifiable, this means that $P(\{\nu \in \cM_1 : \spt\nu \text{ is not strongly $k$-rectifiable}\}) = 1$ for every $P \in \TD(\mu,x)$ at $\mu$ almost every $x$. Note that by choosing appropriate weights, we obtain a self-similar measure of any dimension in $(0,k]$ which is $p$-average $k$-unrectifiable for all $p\in (0,1)$.

(2) On the other hand, any measure $\mu$ supported on a self-similar set $E$ satisfying the strong separation condition of dimension strictly less than $1$ fails to be $0$-average $1$-unrectifiable. Indeed, it follows from self-similarity that for any $x\in E$ and $P\in\TD(\mu,x)$, the support of $P$ almost all measures $\nu$ is contained in a homothetic copy of $E$. On the other hand, any set of upper Minkowski dimension strictly less than $1$ can be covered by a single Lipschitz curve, see e.g. \cite[Lemma 3.1]{BalkaHarangi2014}. 

(3) Given $0 \leq p \leq 1$, it is possible to construct a measure $\mu$ which satisfies precisely
\begin{equation*}
  P(\{ \nu \in \cM_1 : \spt \nu \text{ is not strongly $k$-rectifiable} \}) = p
\end{equation*}
for every $P \in \TD(\mu,x)$ at $\mu$ almost every $x$. The idea of the construction is to \emph{splice} together a $k$-rectifiable measure (such as $\HH^k|_V$ where $V \in G(d,k)$) and the self-similar purely $k$-unrectifiable measure considered in (1) above such that we distribute mass according to the rectifiable measure for portion $1-p$ of scales and the unrectifiable measure for $p$ portion of scales; see Lemma \ref{lma:constructionofunrect} for more details.
\end{example}

For average unrectifiable measures, we obtain an analogous result to Theorem \ref{thm:conical}:

\begin{theorem}\label{thm:unrectifiable}
Suppose that $d \in \N$, $k \in \{ 1,\ldots,d-1 \}$, and $0 \le p < 1$.
If $\mu$ is $p$-average $k$-unrectifiable, then for every $0 < \alpha \le 1$ there exists $0 < \eps < 1$ so that
\begin{equation}\label{eq:conicalporous}
  \liminf_{T \to \infty} \frac{1}{T} \,\lambda\Big(\Big\{ t \in [0,T] :  \inf_{\yli{\theta \in S^{d-1}}{V \in G(d,d-k)}} \frac{\mu(X(x,e^{-t},V,\alpha) \setminus H(x,\theta,\alpha))}{\mu(\overline{B}(x,e^{-t})) } > \epsilon\Big\}\Big) > p
\end{equation}
at $\mu$ almost every $x \in \R^d$.
\end{theorem}

\begin{remark} \label{rem:linkThmsAB}
(1) Since every strongly $k$-rectifiable set is also $k$-rectifiable, if in the definition of average unrectifiability we replace ``strongly rectifiable'' by ``rectifiable'', then Theorem \ref{thm:unrectifiable} continues to hold.

(2) Theorem \ref{thm:unrectifiable} implies the first part of Theorem \ref{thm:conical} for a fixed measure $\mu$. This follows immediately from Lemma \ref{lem:dim-unrectifiability}. Note however that in Theorem \ref{thm:conical}, the value of $\eps$ is independent of $\mu$, while in Theorem \ref{thm:unrectifiable} it is allowed to depend on $\mu$.

(3) In general, the dimension of a measure is strictly smaller than the dimension of its support. Since our definition of average unrectifiability depends only on the support of the measures, Theorem \ref{thm:unrectifiable} reveals that the conical density property in some sense depends on the support of a measure rather than on the measure itself (note that this is not literally true since we need to consider the supports of typical measures for tangent distributions of $\mu$, rather than the support of $\mu$ itself).
\end{remark}

We do not know to what degree Theorem \ref{thm:unrectifiable} provides a complete characterization of the conical density property \eqref{eq:conicalporous}. However, under an additional assumption on the $k$-dimensional densities, we can prove a strong result in the opposite direction, which in particular implies the converse to Theorem \ref{thm:unrectifiable}. We define that a measure $\mu$ is \textit{locally Ahlfors $k$-regular}, if the $k$-\textit{densities} satisfy
\begin{equation} \label{eq:hypothesis-unrect-converseDensity}
0 < \liminf_{r \downarrow 0} \frac{\mu(B(x,r))}{r^k} \leq \limsup_{r \downarrow 0} \frac{\mu(B(x,r))}{r^k} < \infty
\end{equation}
at $\mu$ almost every $x$. Note that in the following result, $\alpha$ is \emph{any} fixed number arbitrarily close to $1$.

\begin{theorem} \label{thm:unrectifiable-converse}
Suppose that $d \in \N$, $k \in \{ 1,\ldots,d-1 \}$, $0 \le p < 1$, $0<\alpha,\eps<1$, and a locally $k$-Ahlfors regular measure $\mu\in\cM$ satisfies
\begin{equation}\label{eq:conicalporous-converse}
  \liminf_{T \to \infty} \frac{1}{T} \,\lambda\Big(\Big\{ t \in [0,T] :  \inf_{V \in G(d,d-k)} \frac{\mu(X(x,e^{-t},V,\alpha))}{\mu(\overline{B}(x,e^{-t})) } > \epsilon\Big\}\Big) > p
\end{equation}
at $\mu$ almost all $x \in \R^d$.
Then for $\mu$ almost all $x$ and all $P\in\TD(\mu,x)$, we have
\[
P(\{ \nu\in\cM:\spt\nu \text{ is not $k$-rectifiable} \}) > p.
\]
In particular, $\mu$ is $p$-average $k$-unrectifiable.
\end{theorem}

\begin{remark}
\label{rmk:unrectconverse}
Theorem \ref{thm:unrectifiable-converse} remains true if the local Ahlfors regularity \eqref{eq:hypothesis-unrect-converseDensity} is replaced by the condition
\begin{equation} \label{eq:hypothesis-unrect-converse}
P(\{\nu\in\cM_1:\nu\ll \cH^k\}) = 1\quad\text{for all } P\in\TD(\mu,x).
\end{equation}
Observe that this condition follows from the local Ahlfors regularity: Indeed, by \cite[Lemma 14.7(1)]{Mattila1995} local Ahlfors regularity yields that for $\mu$ almost every $x$ we have that all tangent measures $\nu \in \Tan(\mu,x)$ are \textit{Ahlfors $k$-regular}, that is, there exists a constant $C>0$ such that
\[
C^{-1}\, r^k \le \nu(B(z,r)) \le C\, r^k\quad\text{for all }z\in\spt\nu \text{ and } r > 0.
\]
Furthermore, if $\nu$ is Ahlfors $k$-regular, then one can readily verify $\nu\ll\cH^k$, and this yields \eqref{eq:hypothesis-unrect-converse}. Since, by Theorem \ref{thm:conical}, the critical dimension for conical densities around $k$-planes is precisely $k$, it is perhaps natural to investigate what happens for measures of this dimension.
\end{remark}

\subsection{Porosity and dimension}\label{porosity}

Porosity is a concept used to quantify the degree of singularity of measures and the size of sets of zero Lebesgue measure. As its name indicates, porosity aims to measure the size and abundance of ``holes'' or ``pores''. For porosity of measures, a ``hole'' is a ball with small (but possibly positive) relative measure. Recall from the Lebesgue density theorem that if a set $E \subset \R^d$ has positive Lebesgue measure, then it contains no holes in the sense that for almost every $x \in E$, if $r>0$ is small, then one cannot find a large part of $B(x,r)$ disjoint from $E$. Thus, the presence of holes of certain relative size at all, or many, scales is indeed a quantitative notion of singularity.

Porosity was introduced by Denjoy \cite{Denjoy1920}. His definition is nowadays called \textit{upper porosity}. Although upper porosity is useful in many connections (see e.g.\ Preiss and Speight \cite{PreissSpeight2014}), one cannot get nontrivial dimension estimates for upper porous sets. The notion of \textit{lower porosity} has arisen from the study of dimensional estimates related, for example, to the boundary behavior of quasiconformal mappings; see the works of Sarvas \cite{Sarvas1975}, Trocenko \cite{Trocenko1981}, Martio and Vuorinen \cite{MartioVuorinen1987}, and V\"ais\"al\"a \cite{Vaisala1987}. In our work, since we are interested in dimension, we consider lower porosity and its modifications. Koskela and Rohde \cite{KoskelaRohde1997} proved that if $f \colon \overline{B}(0,1) \to \R^d$ is quasiconformal and H\"older continuous, then $f(S^{d-1})$ is mean annular porous (see Section \ref{sec:annularporosity} below). Quasiconformal maps may be far from smooth, so this is a deep generalization of the fact that smooth surfaces are lower porous. For mean porosity, it is not required that there are holes present at all scales, but only at a positive proportion of scales.

Porosity has a breadth of applications. An important open problem in complex dynamics is to characterize the rational maps of the Riemann sphere which have Julia sets of full dimension. In the works of Przytycki and Rohde \cite{PrzytyckiRohde1998} and Przytycki and Urba\'nski \cite{PrzytyckiUrbanski2001}, it was shown that certain important classes of Julia sets are mean porous, thereby giving a partial solution to this problem. Porosity has also been applied in the theory of singular integrals; see Chousionis \cite{Chousionis2008b}.

Let us recall some classical notions. We emphasize that we are using \textit{open} balls to define the porosity. This is important in the proof of the closedness of the porosity property; see Lemma \ref{lem:closedcollection}.

\begin{definition}[Porosity] \label{def:porosity}
Let $E \subset \R^d$, $x \in \R^d$, $0 < \alpha \le \tfrac12$, and $r > 0$. We say that $E$ is \emph{$\alpha$-porous at the point $x$ and scale $r$} if there exists $y \in \R^d$ with
$$B(y,\alpha r) \subset \overline{B}(x,r) \setminus E.$$
Moreover, we say that $E$ is \emph{$\alpha$-porous at $x$} if this happens for all small enough $r > 0$, and $E$ is \emph{$\alpha$-porous} if it is $\alpha$-porous at every $x \in E$.

If $\mu$ is a Radon measure and $\eps>0$, then $\mu$ is \emph{$\alpha$-porous at the point $x$ and scale $r$} with threshold $\epsilon$ if there exists $y \in \R^d$ with $B(y,\alpha r) \subset \overline{B}(x,r)$ and
$$\mu(B(y,\alpha r)) \leq \epsilon\mu(\overline{B}(x,r)).$$
If for all small enough $\epsilon > 0$ this happens for all small enough $r > 0$ (depending on $\eps$), then $\mu$ is \emph{$\alpha$-porous at $x$} and if $\mu$ is $\alpha$-porous at $\mu$ almost every $x$, then $\mu$ is \emph{$\alpha$-porous}.
\end{definition}

The connection between porosity and dimension has been under a careful investigation in the last years; see the works of
Mattila \cite{Mattila1988}, Salli \cite{Salli1991}, Eckmann, J\"arvenp\"a\"a, and J\"arvenp\"a\"a \cite{EckmannJarvenpaaJarvenpaa2000},
Beliaev and Smirnov \cite{BeliaevSmirnov2002}, J\"arvenp\"a\"a and J\"arvenp\"a\"a \cite{JarvenpaaJarvenpaa2002},
J\"arvenp\"a\"a, J\"arvenp\"a\"a, K\"aenm\"aki, and Suomala \cite{JarvenpaaJarvenpaaKaenmakiSuomala2005}, Nieminen \cite{Nieminen2006},
Beliaev, J\"arvenp\"a\"a, J\"arvenp\"a\"a, K\"aenm\"aki, Rajala, Smirnov, and Suomala \cite{RajalaSmirnov2007},
J\"arvenp\"a\"a, J\"arvenp\"a\"a, K\"aenm\"aki, Rajala, Rogovin, and Suomala \cite{JarvenpaaJarvenpaaKaenmakiRajalaRogovinSuomala2007},
K\"aenm\"aki and Suomala \cite{KaenmakiSuomala2008,KaenmakiSuomala2004},
Rajala \cite{Rajala2009},
Shmerkin \cite{Shmerkin2012},
Sahlsten, Shmerkin, and Suomala \cite{SahlstenShmerkinSuomala2013},
K\"aenm\"aki, Rajala, and Suomala \cite{KaenmakiRajalaSuomala2012}.
Particular interest has been given to the study of the maximal possible Hausdorff and packing dimensions of a porous set or measure. See K\"aenm\"aki \cite{Kaenmaki2007} and the surveys of J\"{a}rvenp\"{a}\"{a} \cite{Jarvenpaa2010} and Shmerkin \cite{Shmerkin2011} for further background and discussion.

It is easy to see that if a set in $\R^d$ has positive porosity, then its dimension must be strictly less than the dimension of the ambient space. The asymptotical behavior of the dimension drop was described by Trocenko \cite{Trocenko1981} and Martio and Vuorinen \cite{MartioVuorinen1987}. On the other hand, by Mattila \cite{Mattila1988} and Salli \cite{Salli1991}, we know that the dimension of sets $E \subset \R^d$ with porosity $\alpha$ close to the maximum value $\tfrac12$ cannot be much larger than $d-1$: there exists a constant $c>0$ depending only on $d$ such that
\begin{equation} \label{eq:por_dim_estimate}
  \dimp E \leq d-1+\frac{c}{-\log (1-2\alpha)}.
\end{equation}
The above mentioned articles generalize these estimates further for other notions of porosity and to measures by using various methods and techniques.
A consequence of our results in this article is that almost all of these generalizations follow directly from the original Hausdorff dimension estimates for porous sets (which not only came earlier but are conceptually simpler to obtain).

Let us write $\Delta(\alpha)$ for the maximal upper packing dimension of an $\alpha$-porous measure on $\R^d$, that is,
\[
\Delta(\alpha) := \Delta_d(\alpha) = \sup\{\udimp \mu : \mu \text{ is an $\alpha$-porous Radon measure on } \R^d\}.
\]
It has been observed in the earlier works \cite{RajalaSmirnov2007,KaenmakiRajalaSuomala2012,Shmerkin2012,SahlstenShmerkinSuomala2013} that
\[
\sup\{\dimh E : E\subset\R^d \text{ is $\alpha$-porous}\}
\]
is, up to certain constants, asymptotically equal to $\Delta(\alpha)$ when $\alpha$ is either close to its minimum or maximum value. In other words, whether we consider sets or measures, Hausdorff or packing dimension, the largest possible dimension turns out to be the same. Moreover, the examples illustrating the sharpness of these results are always self-similar.
Using fractal distributions we are able to show that  this is a general phenomenon with a conceptual explanation.

Concerning the intermediate values of $\alpha$, Salli \cite[Remark 3.8.2(1)]{Salli1991} proved that, for each $0 < \alpha \le 1/2$,
\begin{equation} \label{eq:salli}
  \sup\{ \dimh E : E \subset \R \text{ is $\alpha$-porous} \} = \dimh C_\alpha = \frac{\log 2}{\log(2-2\alpha) - \log(1-2\alpha)},
\end{equation}
where $C_\alpha \subset \R$ is the standard $\frac{1-2\alpha}{2-2\alpha}$-Cantor set. This result is intrinsic to the real line and the proof does not generalize to higher dimensions. Besides generalizing this to porous measures, the next result also says that largest possible dimensions obtained from $\alpha$-porous sets and measures are the same for all values of $\alpha$, in any dimension. In particular, the packing dimension estimates for $\alpha$-porous measures follow immediately from the Hausdorff dimension estimates for $\alpha$-porous sets when $\alpha$ is either close to its minimum or maximum value. Moreover, extremal measures (that is, $\alpha$-porous measures of maximal dimension) exist, and can be chosen to be \textit{uniformly scaling}, which is a generalized version of self-similarity inspired by ergodic theory, see Definition \ref{def:USM} below.

\begin{theorem} \label{thm:mainporosity}
For any $0 < \alpha \leq 1/2$ we have
\[
\Delta_d(\alpha) = \sup\{ \dimh E : E\subset\R^d \text{ is $\alpha$-porous} \}.
\]
Moreover, the supremum in the definition of $\Delta_d(\alpha)$ is attained by an exact-dimensional measure, which furthermore is uniformly scaling.
\end{theorem}

Theorem \ref{thm:mainporosity} is loosely inspired by a result of Furstenberg on \emph{galleries} of sets \cite[Theorem 5.1]{Furstenberg2008}: a gallery is a collection of compact subsets of $\R^d$ which is closed under passing to subsets, under magnifications, and under limits in the Hausdorff metric. Furstenberg proved that for any gallery there is a measure supported on a set of the gallery whose Hausdorff dimension equals the supremum of the Assouad (in particular, also Hausdorff or packing) dimensions of sets in the gallery. Note that no direct application of this result is possible in our context, since porous sets are far from forming a gallery, and porous measures need not be supported on porous sets. Instead, the idea of the proof is that if we start with a porous measure and take a tangent distribution at a typical point and then consider a typical measure for this distribution, it is in fact supported on a uniformly porous set.

The relationship between porosity and tangents has been previously studied by Mera and Mor\'{a}n \cite{MeraMoran2001}, who proved that when magnifying a porous measures under suitable assumptions, the tangents we see have set theoretical holes in their support, and by Orponen and Sahlsten \cite{OrponenSahlsten2012}, who further observed that this property does not appear for general measures if the notion of porosity is too weak (i.e.\ upper porosity). When the porosity assumption is stronger, such as lower or mean porosity, holes should exist for many tangent measures. Verifying this directly is far from straightforward, but equipped with the machinery of fractal distributions it is a rather easy task. Indeed, it is almost immediate that tangent measures to porous measures have holes \emph{at the origin}. The quasi-Palm property can then be invoked to guarantee the existence of holes around typical points of the support.

A natural problem that as far as we know has not been addressed in dimensions $d\ge 2$ concerns the regularity of the function $\alpha\mapsto\Delta(\alpha)$. One might suspect that this map is, at the very least, continuous and strictly decreasing (notice that $\Delta$ is clearly non-increasing: an $\alpha$-porous set is $\alpha'$-porous for all $\alpha' \leq \alpha$). We recall that in the real line this follows from \eqref{eq:salli} and Theorem \ref{thm:mainporosity}. Although a full proof or disproof in arbitrary dimension appears to require new geometric ideas beyond the scope of the fractal distribution machinery, we take a first step by proving upper semicontinuity:

\begin{theorem}\label{thm:semic}
The function $\alpha \mapsto \Delta(\alpha)$ is upper semicontinuous.
\end{theorem}

\subsection{Mean porosity and dimension} \label{sec:meanporosity}

Theorem \ref{thm:mainporosity} only concerns sets and measures for which \textit{all} scales contain pores. However, there are natural examples that exhibit such behavior on a positive proportion of scales, such as quasiconformal images of the unit sphere $S^{d-1}$; see \cite{KoskelaRohde1997}. The notion of mean porosity was introduced in view of these natural examples; see, for example, Koskela and Rohde \cite{KoskelaRohde1997} and Beliaev and Smirnov \cite{BeliaevSmirnov2002}.

\begin{definition}[Mean porosity] \label{def:mean-poro}
Let $0 < p \le 1$. A set $E \subset \R^d$ is \emph{$p$-mean $\alpha$-porous at $x$} if
  \begin{align*}
  \liminf_{T \to \infty} \frac{1}{T}\, \lambda(\{ t \in [0,T] : \text{there is } y \in \R^d \text{ with } B(y,\alpha e^{-t}) \subset \overline{B}(x,e^{-t}) \setminus E \}) \ge p.
  \end{align*}
  Moreover, we say that $E$ is \emph{$p$-mean $\alpha$-porous} if this holds for all $x$.

  A Radon measure $\mu$ is \emph{$p$-mean $\alpha$-porous at $x$} if for all small enough $\epsilon > 0$ we have
    \begin{equation}
    \begin{split}
  \label{meanpor}\liminf_{T \to \infty} \frac{1}{T} \,\lambda(\{ t \in [0,T] : \;&\text{there is } y \in \R^d \text{ with } B(y,\alpha e^{-t}) \subset \overline{B}(x,e^{-t}) \\ &\text{and } \mu(B(y,\alpha e^{-t})) \leq \epsilon\mu(\overline{B}(x,e^{-t})) \}) \ge p.
  \end{split}
  \end{equation}
  If this happens at $\mu$ almost every $x$, then $\mu$ is \emph{$p$-mean $\alpha$-porous}.
\end{definition}

\begin{remark}
The usual definition of mean porosity counts only dyadic scales. Our definition is invariant under homotheties and is not tied to a base. All the previous results on dimensions of mean porous measures continue to hold with this definition, up to the values of the constants. See the discussion in \cite[Section 6.1]{Shmerkin2012}.
\end{remark}

The upper bound for the dimension of sets and measures with $p$-mean $\alpha$-porosity close to its maximum value cannot be much larger than $d-p$; see Sahlsten, Shmerkin, and Suomala \cite[Theorem 1.3]{SahlstenShmerkinSuomala2013}. The natural analogue of $\Delta(\alpha)$ for mean porosity is
\[
\Delta(\alpha,p) := \Delta_d(\alpha,p) = \sup\{\udimp \mu : \mu \text{ is a $p$-mean $\alpha$-porous Radon measure on } \R^d\}.
\]
Inspecting the proof of Theorem \ref{thm:mainporosity}, we observe that
\begin{equation} \label{eq:1-mean_poro_is_poro}
  \Delta(\alpha,1) = \sup\{\Hd E : E\subset\R^d \text{ is $\alpha$-porous}\}.
\end{equation}
This holds basically because tangent distributions arising from the scenery flow are defined as limits of Ces\`aro means, so a zero density set of scales does not affect the outcome. It is natural to ask if
\[
\Delta(\alpha,p) = \sup\{\Hd E : E\subset\R^d \text{ is $p$-mean $\alpha$-porous}\}.
\]
This equality was verified asymptotically (up to the value of certain constants) in the limits $\alpha\to 0$ (see \cite{Shmerkin2012}) and $\alpha\to 1/2$ (see \cite{RajalaSmirnov2007}). Unfortunately, this does not seem to follow by directly applying the machinery of fractal distributions. However, we do get a partial result: if we consider only porosity of measures, then the supremum of measures satisfying a quantitative mean porosity condition is the same whether we consider Hausdorff or packing dimension, either in their lower or upper versions. We underline that this is far from obvious, since in the development of the theory, the progression from the Hausdorff dimension estimates to the packing dimension estimates when $\alpha$ is close to its maximum value was the most difficult.

\begin{theorem}\label{thm:meanporosity}
For any $0 < \alpha \le 1/2$ and $0 < p \leq 1$ we have
\[
  \Delta_d(\alpha,p) = \sup\{\ldimh \mu : \mu \text{ is a $p$-mean $\alpha$-porous Radon measure on } \R^d\} .
\]
Moreover, the supremum in the definition of $\Delta_d(\alpha,p)$ is attained by an exact dimensional measure, which furthermore is uniformly scaling. Also,
\[
 \Delta_d(\alpha,p)  \ge p\Delta_d(\alpha) + (1-p)d.
\]
\end{theorem}

\begin{remark} \label{rem:porosity-remarks}
  (1) We do not know whether $p \mapsto \Delta(\alpha,p)$ is in fact affine; this would be consistent with the results in \cite{RajalaSmirnov2007} and \cite{Shmerkin2012}. It appears that answering this question requires understanding how porous scales are distributed in $p$-mean $\alpha$-porous measures of dimension close to $\Delta(\alpha,p)$. Also, any continuity properties of $\alpha \mapsto \Delta(\alpha,p)$ for $0<p<1$ remain open.

  (2) The concepts of porosity and mean porosity given in Definitions \ref{def:porosity} and \ref{def:mean-poro} are not suitable tools to describe sets of dimension less than $d-1$. For example, each $V \in G(d,d-1)$ has maximal porosity. For this reason K\"aenm\"aki and Suomala \cite{KaenmakiSuomala2008,KaenmakiSuomala2004} introduced the concept of $k$-porosity. Such $k$-porous sets are required to have holes in $k$ orthogonal directions near each of its points in every small scale. The main feature of this property is that if $p$-mean $k$-porosity is close to its maximum value, then the dimension cannot be much larger than $d-pk$. The mean version of the definition is explicitly given in Sahlsten, Shmerkin, and Suomala \cite{SahlstenShmerkinSuomala2013}. Inspecting the proofs of Theorems \ref{thm:mainporosity}--\ref{thm:meanporosity}, it is evident that the results generalize also to this case.
\end{remark}

\subsection{Annular porosity}\label{sec:annularporosity}

The concept of mean porosity from the previous section was historically not the first one to be introduced. Koskela and Rohde \cite{KoskelaRohde1997} defined and studied a different quantitative notion of porosity of sets. In this version, more information about the actual location of the hole is required. To distinguish this definition from the classical one, we call it \textit{annular porosity}: the central points of the pores are required to lie in a certain annulus. We also introduce annular porosity for measures.

\begin{definition}[Annular porosity] \label{def:annular-porosity}
Let $0 < \alpha,\roo \le 1$ and write $c = (1+\roo)^{-1}$. Let $A(x,cr,r)$ be the closed annulus $\overline{B}(x,r) \setminus B(x,cr)$. We say that $E$ is \emph{$\roo$-annular $\alpha$-porous at the point $x$ and scale  $r$}, if there exists $y \in A(x,cr,r)$ with
$$B(y,\alpha \roo |x-y|) \cap E = \emptyset.$$
Moreover, we say that $E$ is \emph{$\roo$-annular $\alpha$-porous at $x$} if this happens for all small enough $r > 0$ and $E$ is \emph{$\roo$-annular $\alpha$-porous} if it is $\roo$-annular $\alpha$-porous at every $x \in E$.

If $\mu$ is a Radon measure and $\eps > 0$, then $\mu$ is \emph{$\roo$-annular $\alpha$-porous at the point $x$ and scale $r$} with threshold $\eps$, if there exists $y \in A(x,cr,r)$ with
$$\mu(B(y,\alpha \roo |x-y|)) \leq \epsilon\mu(\overline{B}(x,r)).$$
If for all small enough $\epsilon > 0$ this happens for all small enough $r > 0$ (depending on $\eps$), then $\mu$ is \emph{$\roo$-annular $\alpha$-porous at $x$} and if $\mu$ is $\roo$-annular $\alpha$-porous at $\mu$ almost every $x$, then $\mu$ is \emph{$\roo$-annular $\alpha$-porous}.
Finally, mean annular porosity for sets and measures is defined analogously to the classical case.
\end{definition}

\begin{figure}[t!]
\label{fig:annularpic}
\includegraphics[scale=0.7]{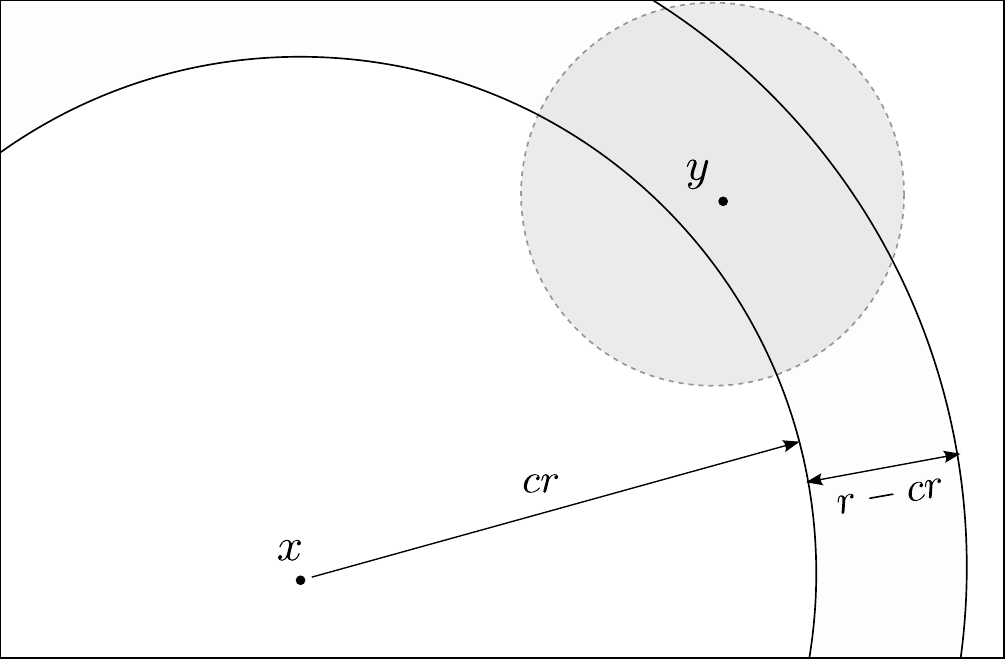}
\caption{The center $y$ of the $\alpha$-hole of $E$ or $\mu$ can only be chosen from the annulus $A(x,cr,r)$.}
\end{figure}

Koskela and Rohde found bounds for the packing dimension of mean annular porous sets. In \cite[Theorem 2.1]{KoskelaRohde1997}, they proved that if $E \subset \R^d$ is a $p$-mean $\roo$-annular $\alpha$-porous set, then
\begin{equation} \label{eq:KR}
  \dimp E \le d - Cp\roo^{d-1}\alpha^d
\end{equation}
where $C$ depends only on $d$. They also showed that the exponents in the estimate are the best possible ones.
If $\roo$ is close to zero, then also the width of the annulus $A(x,cr,r)$ is close to zero. Thus for small $\alpha$ the annular porosity requires that each ball contains a lot of pores roughly of the same size. Therefore, if $\alpha$ is fixed, we get better asymptotic behavior for the dimension of $\roo$-annular $\alpha$-porous sets. As a particular application, Koskela and Rohde showed in \cite[Corollary 3.2]{KoskelaRohde1997} that if $f \colon \overline{B}(0,1) \to \R^d$ is $K$-quasiconformal and $\roo$-H\"older continuous, then the Minkowski dimension of $f(S^{d-1})$ is at most $d-C\roo^{d-1}$ where $C$ depends only on $d$ and $K$. For basic properties of quasiconformal mappings, the reader is referred to the book of Ahlfors \cite{Ahlfors2006}.

For measures, such bounds have not yet been found. We again reduce the problem back to the set case. We write
$$\Delta^\roo(\alpha) := \Delta_d^\roo(\alpha) = \sup\{ \udimp \mu : \mu \text{ is $\roo$-annular $\alpha$-porous on $\R^d$} \}$$
and
$$\Delta^\roo(\alpha,p) := \Delta_d^\roo(\alpha,p) = \sup\{ \udimp \mu : \mu \text{ is $p$-mean $\roo$-annular $\alpha$-porous on $\R^d$} \}.$$
The following result relates the dimension of $p$-mean $\roo$-annular $\alpha$-porous measures back to the dimensions of $p$-mean $\roo$-annular $\alpha$-porous measures of sets, thereby extending the estimate \eqref{eq:KR} of Koskela and Rohde from sets to measures.

\begin{theorem}\label{thm:annularporosity}
For any $0 < \alpha,\roo \leq 1$ and $0<p\le 1$ we have
\[
\Delta_d^\roo(\alpha) = \sup\{ \dimh E : E \subset \R^d \text{ is $\roo$-annular $\alpha$-porous} \}
\]
and
\[
\Delta_d^\roo(\alpha,p) = \sup\{ \ldimh \mu : \mu \text{ is $p$-mean $\roo$-annular $\alpha$-porous on $\R^d$} \} \geq p\Delta_d^\roo(\alpha) + (1-p)d.
\]
Moreover, the suprema in the definitions of $\Delta_d^\roo(\alpha)$ and $\Delta_d^\roo(\alpha,p)$ are attained by exact dimensional measures, which furthermore are uniformly scaling.
\end{theorem}

The proof of Theorem \ref{thm:annularporosity} is very similar to the proofs of Theorems \ref{thm:mainporosity} and \ref{thm:meanporosity}, only the geometric details coming from the location of the hole differ.

\subsection{General norms}

Although our main interest is in the Euclidean metric, we remark that our results can be generalized to more general norms on Euclidean space. This is due to the fact that the machinery of scenery flows is independent of the choice of the norm; see \cite[Appendix A]{KaenmakiSahlstenShmerkin2014}. The condition we need to impose to the shape of the norm is that the unit sphere is a $C^1$ manifold that does not contain line segments. This extra requirement is only needed in the proof of Theorem \ref{thm:nomass} below to allow the inductive argument to go through. Some conical density and porosity results independent of the choice of norm have been obtained before. For example, Rajala \cite[Corollary 2.3]{Rajala2009} proved that the dimension estimate \eqref{eq:por_dim_estimate} for porous sets is also independent of the choice of norm and also holds in certain metric spaces.

\section{The scenery flow and fractal distributions} \label{sec:fractaldistributions}

Before we set out to proving the results, we recall the main definitions and results for the scenery flow and fractal distributions from Hochman \cite{Hochman2010}, and some enhancements from K\"aenm\"aki, Sahlsten, and Shmerkin \cite{KaenmakiSahlstenShmerkin2014}, which are required in our geometric investigations. We also introduce some new results.

\subsection{Ergodic theory of flows}

Let us recall some basic facts concerning the dynamics of flows; see for example the books by Einsiedler and Ward \cite{EinsiedlerWard2011} and Walters \cite{Walters1982}. Let $X$ be a metric space and write $\R_+ = [0,\infty)$. A (one-sided) \textit{flow} is a family $(F_t)_{t \in \R_+}$ of maps $F_t \colon X \to X$ for which
$$F_{t+t'} = F_{t} \circ F_{t'}, \quad t,t' \in \R_+.$$
In other words, $(F_t)$ is an additive $\R_+$ action on $X$. If $(X,\cB,P)$ is a probability space, then we say that $P$ is $F_t$ \textit{invariant} if $F_t P = P$ for all $t \geq 0$. In this case, we call $(X,\cB,P,(F_t)_{t \in \R_+})$ a \textit{measure preserving flow}.
We say that a measure preserving flow is \textit{ergodic}, if for all $t \geq 0$ the measure $P$ is ergodic with respect to the transformation $F_t \colon X \to X$, that is, for all $F_t$ invariant sets $A \in \cB$ we have $P(A) \in \{ 0,1 \}$. A set $A \in \cB$ is $F_t$ \emph{invariant} if $P(F_t^{-1} A \triangle A)=0$.

The ergodicity of a flow guarantees that the expectation of an observable $f \colon X \to \R$ can be approximated by averages $\frac{1}{T}\int_0^T f(F_t x) \, \mathrm{d}t$ where the integral is the usual Lebesgue integral. This is the famous \textit{Birkhoff's ergodic theorem}: if $(X,\cB,P,(F_t)_{t \in \R_+})$ is an ergodic measure preserving flow, then for a $P$ integrable function $f \colon X \to \R$ we have
$$\lim_{T \to \infty}  \frac{1}{T} \int_0^T f(F_t x) \, \mathrm{d}t = \int f \, \mathrm{d}P$$
at $P$ almost every $x \in X$.

A general $F_t$ invariant measure $P$ can be decomposed into component measures $P_\omega$, $\omega \sim P$, such that
$$P = \int P_\omega \, \mathrm{d}P(\omega)$$
and the measures $P_\omega$ on $X$ are $F_t$ invariant and ergodic. This is known as the \textit{ergodic decomposition} of $P$ and the measures $P_\omega$ are the \textit{ergodic components}. Moreover, this decomposition is unique up to $P$ measure zero sets.

\subsection{Scenery flow}
We will now define the \textit{scenery flow}. The idea behind it is to record the statistics of the magnifications of Radon measures $\mu \in \cM$ around a point in their support. We follow notation used in Hochman \cite{Hochman2010} and K\"aenm\"aki, Sahlsten, and Shmerkin \cite{KaenmakiSahlstenShmerkin2014}.

\begin{definition}[Scenery flow]
Let $\mu \in \cM_1$ with $0 \in \spt \mu$ and $t \in \R$. Define the $e^{-t}$ scale \textit{magnification} $S_t \mu \in \cM_1$ of $\mu$ at $0$ by
$$S_t\mu(A) = \frac{\mu(e^{-t}A)}{\mu(\overline{B}(0,e^{-t}))}, \quad A \subset B_1.$$
Due to the exponential scaling, $(S_t)_{t \in \R_+}$ is a flow in the space
$$\cM_1^* := \{\mu \in \cM_1 : 0 \in \spt \mu\}$$
and we call it the \textit{scenery flow} at $0$.
\end{definition}

We remark that our $S_t$ is denoted by $S_t^\square$ in \cite{Hochman2010} and \cite{KaenmakiSahlstenShmerkin2014}. We note that the action $S_t$ is discontinuous (at measures $\mu$ for which $\mu(\partial B(0,r))>0$ for some $r\in (0,1)$) and the space $\cM_1^* \subset \cM_1$ is Borel but is not closed. Nevertheless, the philosophy behind several of the results which we will recall is that, in practice, the scenery flow behaves in a very similar way to a continuous flow on a compact metric space.

If we have an arbitrary Radon measure $\mu \in \cM$ and $x \in \spt\mu$ we want to consider the scaling dynamics when magnifying around $x$. For this purpose, we shift the space back to the origin. Let $T_x \mu \in \cM$ be the \textit{translation} defined by
\begin{align*}
T_x \mu(A) = \mu(A+x).
\end{align*}

\begin{definition}[Scenery and tangent measures at $x$]
Given $\mu\in\MM$ and $x\in\spt\mu$, we consider the one-parameter family $(\mu_{x,t})_{t \geq 0}$ in $\cM_1$ defined by
\[
\mu_{x,t} := S_t(T_x\mu)
\]
and call it the \emph{scenery of $\mu$ at $x$}. Accumulation points of this scenery in $\cM_1$ will be called \emph{tangent measures} of $\mu$ at $x$ and the family of tangent measures of $\mu$ at $x$ is denoted by $\Tan(\mu,x) \subset \cM_1$.
\end{definition}

\begin{remark}
We deviate slightly from the usual definition of tangent measures, which corresponds to taking weak limits of unrestricted blow-ups.
\end{remark}

As noted in the introduction, one of the main ideas of this work is that, as far as certain properties of a measure are concerned, including their dimensions, the ``correct'' tangent structure to consider is not a single limit of $\mu_{x,t}$ along some subsequence, but the whole statistics of the scenery $\mu_{x,t}$ as $t\to\infty$.

\begin{definition}[Scenery and tangent distributions]
The \emph{scenery distribution} of $\mu$ up to time $T > 0$ at $x \in \spt \mu$ is defined by
$$
  \langle \mu \rangle_{x,T} := \frac{1}{T} \int_0^T \delta_{\mu_{x,t}} \,\mathrm{d}t.
$$
We call any weak limit of $\langle \mu \rangle_{x,T}$ for $T \to \infty$ in $\cP(\cM_1)$ a \emph{tangent distribution} of $\mu$ at $x$. The family of tangent distributions of $\mu$ at $x$ will be denoted by $\TD(\mu,x)$. Notice that the support of each $P \in \TD(\mu,x)$ is contained in $\Tan(\mu,x)$.
\end{definition}

\begin{remark} \label{rem:TD-compact}
In the above definition, the integration makes sense since we are on a convex subset of a topological linear space. If the limit is unique, then, intuitively, it means that the collection of views $\mu_{x,t}$ will have well defined statistics when zooming into smaller and smaller neighbourhoods of $x$. Since $\TD(\mu,x)$ is defined as a set of accumulation points in a compact space $\cP(\cM_1)$, the subspace $\TD(\mu,x)$ is always non-empty and compact at $x \in \spt \mu$.
\end{remark}

\begin{remark} \label{rem:scenery_geom}
Scenery distributions can be used to restate the conical density properties and the definition of mean porosities. For example, the conical density property \eqref{eq:lowerconicalporous} is equivalent to
  \begin{equation*}
    \liminf_{T \to \infty}\, \la \mu \ra_{x,T}\Big(\Big\{ \nu \in \cM_1 : \inf_{\yli{\theta \in S^{d-1}}{V \in G(d,d-k)}} \nu(X(0,1,V,\alpha) \setminus H(0,\theta,\alpha)) > \eps\Big\} \Big) \ge \frac{s-k}{d-k},
  \end{equation*}
and a measure $\mu$ is $p$-mean $\alpha$-porous if and only if for all $\eps>0$ we have
  \begin{equation*}
    \liminf_{T \to \infty}\, \la \mu \ra_{x,T}( \{ \nu \in \cM_1 : \nu(B(y,\alpha)) \le \eps \text{ for some } y \in \overline{B}(0,1-\alpha) \}) \ge p.
  \end{equation*}
  This is the main link between the geometric problems we consider and the scenery flow.
\end{remark}

\subsection{Fractal distributions}

As with usual tangents, tangent measures and distributions enjoy some kind of spatial invariance. Preiss proved in his seminal paper \cite{Preiss1987} that at almost every point, tangent measures to tangent measures are again tangent measures; this has been a significant feature in the applications of tangent measures. A result of this kind for tangent distributions was obtained by Hochman \cite{Hochman2010}, and it will be a key tool in our geometric applications. In order to state it, we need some additional definitions.


\begin{definition}[Fractal distributions] \label{def:FD}
We say that the distribution $P$ on $\cM_1$ is
\begin{enumerate}
\item \emph{scale invariant} if it is invariant under the action of the semigroup $S_t$, that is,
\[
  P(S_t^{-1} \AA) = P(\AA)
\]
for all Borel sets $\AA \subset \PP(\cM_1)$ and all $t \geq 0$.
\item \emph{quasi-Palm} if for any Borel set $\cA \subset \cM_1$ with $P(\cA) = 1$ we have that $P$ almost every $\nu \in \cA$ satisfies
$$\nu_{z,t}  \in \cA$$
for $\nu$ almost every $z \in \R^d$ with $\overline{B}(x,e^{-t}) \subset B_1$.
\item a \emph{fractal distribution} (FD) if it is scale invariant and quasi-Palm.
\item an \emph{ergodic fractal distribution} (EFD) if it is a fractal distribution and it is ergodic with respect to $S_t$.
\end{enumerate}
Write $\FD$ and $\EFD$ for the set of all fractal distributions and ergodic fractal distributions, respectively.
\end{definition}

\begin{remark} \label{rem:norm}
(1) We note that the above definition of quasi-Palm is different than the one employed by Hochman \cite{Hochman2010} and K\"aenm\"aki, Sahlsten, and Shmerkin \cite{KaenmakiSahlstenShmerkin2014}; we invoke the notion used by Hochman and Shmerkin \cite[Section 4.2]{HochmanShmerkin2014} under the name $S$-quasi-Palm. The main difference is that, in the above definition, the action $S_t$ is restricted to the unit ball, while in the quoted papers the action is on Radon measures of $\R^d$. This restriction makes it necessary to consider the magnifications $\nu_{z,t}$ (rather than just the translations $T_z\nu$ as in \cite{Hochman2010}). For this reason our definition of quasi-Palm is consistent with that of \cite{Hochman2010} only when the distribution is also $S_t$-invariant, that is, a FD. Since we will only apply the quasi-Palm property to FDs, and our restricted FDs are in canonical one-to-one correspondence with the unrestricted ones as shown in \cite[Lemma 3.1]{Hochman2010}, this will cause no problem. See \cite[Section 4.2]{HochmanShmerkin2014} for further discussion on the link between the two alternative definitions of quasi-Palm.

(2) Hochman \cite{Hochman2010} and K\"aenm\"aki, Sahlsten, and Shmerkin \cite{KaenmakiSahlstenShmerkin2014} used the $L^\infty$ norm instead of the Euclidean norm. The reason for this is that it allows an easier link between fractal distributions and CP processes. Many of the results concerning fractal distributions are proved by using CP processes which are Markov processes on the dyadic scaling sceneries of a measure introduced by Furstenberg in \cite{Furstenberg1970,Furstenberg2008}. However, the results in \cite{Hochman2010} and \cite{KaenmakiSahlstenShmerkin2014} are independent of the choice of the norm; see the discussion in \cite[Appendix A]{KaenmakiSahlstenShmerkin2014}.
\end{remark}

We start with some basic properties of fractal distributions that can be readily checked from the definitions. They will be used throughout the paper without further reference.

\begin{lemma}
\begin{enumerate}
\item If $P_1,\ldots,P_n$ are fractal distributions and $(q_1,\ldots,q_n)$ is a probability vector, then $q_1 P_1+\cdots+q_n P_n$ is a fractal distribution.
\item Let $\mu$ be the normalized restriction to $B_1$ of one of the following measures: Lebesgue measure $\LL^d$, the mass at zero $\delta_0$, or the restriction of Hausdorff measure $\cH^k$ to some plane $V\in G(d,k)$, where $k \in \{ 1,\ldots,d-1 \}$. Then $P=\delta_\mu$ is a fractal distribution.
\end{enumerate}
\end{lemma}

The result of Hochman \cite[Theorem 1.7]{Hochman2010} shows that typical tangent distributions enjoy an additional spatial invariance:

\begin{theorem}\label{thm:TDs-are-FDs}
For any $\mu\in\MM$ and $\mu$ almost every $x$, all tangent distributions at $x$ are fractal distributions.
\end{theorem}

Notice that as the action $S_t$ is discontinuous, even the scale invariance of tangent distributions or the fact that they are supported on $\cM_1^*$ are not immediate, though they are perhaps expected. The most interesting part in the above theorem is that a typical tangent distribution satisfies the quasi-Palm property.

The quasi-Palm property is also preserved when passing to the ergodic components:

\begin{theorem}\label{thm:ergodic-components-of-FDs-are-FDs}
The ergodic components of fractal distributions are ergodic fractal distributions.
\end{theorem}

See Hochman \cite[Theorem 1.3]{Hochman2010} for the proof. The above theorem is an instance of the principle that although fractal distributions are defined in terms of seemingly strong geometric properties, the family of fractal distributions is in fact very robust. The following result of K\"aenm\"aki, Sahlsten, and Shmerkin \cite[Theorem A]{KaenmakiSahlstenShmerkin2014} is another manifestation.

\begin{theorem} \label{thm:Fds-are-closed}
  The family of fractal distributions is compact.
\end{theorem}

In geometric considerations, we usually construct a fractal distribution satisfying certain property. We often want to transfer that property back to a measure. This leads us to the concept of generated distributions.

\begin{definition}[Uniformly scaling measures] \label{def:USM}
We say that a measure $\mu$ \emph{generates} a distribution $P$ at $x$ if
$\TD(\mu,x)=\{P\}.$
Furthermore, $\mu$ \emph{generates $P$} if it generates $P$ at $\mu$ almost every point. In this case, we say that $\mu$ is a \emph{uniformly scaling measure} (USM).
\end{definition}

One simple example of a uniformly scaling measure is $\mu=\mathcal{H}^k|_V$, where $V$ is a $k$-plane; it generates the distribution $P=\delta_\nu$, where $\nu$ is the normalized restriction of $\mu$ to $B_1$. Another example is the occupation measure of Brownian motion in dimension $d\ge 3$, with $P$ the distribution of the occupation measure of a Brownian motion started at $0$, normalized and restricted to the unit ball. This follows from the self-similarity of Brownian motion; see \cite[Theorem 3.1]{Gavish2011}. Further examples include self-similar measures under a suitable separation condition \cite{Hochman2010}, and measures invariant under $\beta$ shifts $x\mapsto \beta x\bmod 1$ (see \cite{Hochman2012} for the integer case, and \cite{HochmanShmerkin2014} for general $\beta>1$).

If $P$ is ergodic, then, as a consequence of the ergodic theorem, $P$ almost every measure generates $P$; see Hochman \cite[Theorem 3.90]{Hochman2010}.

\begin{theorem}
\label{thm:ergodicgenerates}
If $P$ is an ergodic fractal distribution, then $P$ almost every $\mu$ is a uniformly scaling measure generating $P$.
\end{theorem}

A rather technical argument based on careful \textit{splicing of scales} shows that even a non-ergodic fractal distribution can be generated by a uniformly scaling measure; see K\"aenm\"aki, Sahlsten, and Shmerkin \cite[Theorem C]{KaenmakiSahlstenShmerkin2014}.

\begin{theorem} \label{thm:measure-generating-FD}
For any fractal distribution $P$, there exists a uniformly scaling measure $\mu$ generating $P$.
\end{theorem}

Another useful fact is that the property of being uniformly scaling is preserved under normalized restrictions. If $\mu \in \cM$ and $A \subset \R^d$ is a Borel set with $0<\mu(A)<\infty$, then it is a consequence of the Besicovitch density point theorem \cite[Corollary 2.14]{Mattila1995} that if $\mu$ generates $P$, then the normalized restriction measure $\mu_A= \mu(A)^{-1}\mu|_A$ also generates $P$. More generally, if $\nu,\mu\in\cM$, $\nu\ll\mu$ and $\mu$ generates $P$, then $\nu$ also generates $P$. See Hochman \cite[Propositions 3.7 and 3.8]{Hochman2010}. In fact, the same arguments yield the following more general result.

\begin{theorem}\label{thm:restriction}
If $\mu\in\cM$ and $0<\mu(A)<\infty$, then for $\mu$ almost all $x\in A$ we have $\TD(\mu,x)=\TD(\mu_A,x)$. More generally, if $\nu\in\cM$ and $\nu\ll\mu$, then $\TD(\mu,x)=\TD(\nu,x)$ at $\nu$ almost all $x$. In particular, if $\mu$ is a uniformly scaling measure generating $P$, then $\mu_A$ is a uniformly scaling measure generating $P$.
\end{theorem}

\subsection{Dimension of fractal distributions}

In this section, we discuss the behaviour of fractal distributions with respect to \textit{dimension}. A first indication of the regularity of FDs is that almost every measure is exact-dimensional.

\begin{proposition} \label{prop:dimension-FDs}
If $P$ is a fractal distribution, then $P$ almost all measures are exact-dimensional. Furthermore, if $P$ is ergodic, then the value of the dimension is $P$ almost everywhere constant, and given, for any $r \in (0,1)$, by
\[
\int \frac{\log \mu(B(0,r))}{\log r} \,\mathrm{d}P(\mu).
\]
\end{proposition}

This is Hochman \cite[Lemma 1.18]{Hochman2010} and an application of Theorem \ref{thm:ergodic-components-of-FDs-are-FDs}, the ergodic decomposition of fractal distributions. Proposition \ref{prop:dimension-FDs} motivates the following definition of \textit{dimension} of a fractal distribution:

\begin{definition}[Dimension of fractal distributions]  \label{def:dimension} The \emph{dimension} of a fractal distribution $P$ is
\[
\dim P = \int \dim\mu\, \mathrm{d}P(\mu) = \int \frac{\log \mu(B(0,r))}{\log r}\, \mathrm{d}P(\mu).
\]
for any $r \in (0,1)$. Here the second equality follows from Proposition \ref{prop:dimension-FDs} and the ergodic decomposition.
\end{definition}

Since the dimension of a measure on $\R^d$ varies between $0$ and $d$, so does the dimension of a fractal distribution. Although there are many measures $\mu$ satisfying either $\dim\mu=0$ or $\dim\mu=d$, there is only one fractal distribution attaining each of these extreme values for the dimension.

\begin{lemma}\label{lem:onlydim}
If $P$ is a fractal distribution, then
\begin{itemize}
  \item[(1)] $\dim P = 0$ if and only if $P = \delta_{\delta_0}$,
  \item[(2)] $\dim P = d$ if and only if $P = \delta_{\overline{\LL}^d}$.
\end{itemize}
\end{lemma}

See \cite[Propositions 6.4 and 6.5]{Hochman2010} for the proof.

Hausdorff and packing dimensions are highly discontinuous on measures $\mu \in \cM_1$. For fractal distributions we obtain continuity:

\begin{lemma} \label{lem:continuity-of-FD-dim}
The function $P \mapsto \dim P$ defined on $\FD$ is continuous.
\end{lemma}

\begin{proof}
Although the function $\mu\mapsto \log \mu(B(0,r))$ is in general discontinuous, a given $\mu$ is a discontinuity point for at most countably many $r\in (0,1)$, hence by the dominated convergence theorem, the function
\[
F(\mu) = \int_0^1 \frac{\log \mu(B(0,r))}{\log r} \,\mathrm{d}r
\]
is continuous on $\mathcal{M}_1^*$ and, by Fubini, $\dim P=\int F\, \mathrm{d}P.$
\end{proof}

Intuitively, the local dimensions of a measure should not be affected by the geometry of the measure on a density zero set of scales. This can be formalized using \textit{local entropy averages} (see for example \cite{HochmanShmerkin2012}). Thus heuristically one could expect that tangent distributions, defined as time averages, should encode all information on dimensions.  The following observation, proved by Hochman \cite[Proposition 1.19]{Hochman2010}, shows that this is indeed the case.

\begin{theorem} \label{thm:local-tangent}
Given a measure $\mu \in \cM$, for $\mu$ almost all $x \in \R^d$ the local dimensions
\begin{align*}
\udimloc(\mu,x) &= \sup\{ \dim P: P\in \TD(\mu,x) \cap \FD \}, \\
\ldimloc(\mu,x) &= \inf\{ \dim P: P\in \TD(\mu,x) \cap \FD \}.
\end{align*}
In particular, if $\mu$ is a USM generating a fractal distribution $P$, then $\mu$ is exact dimensional and
$$\dim\mu = \dim P.$$
\end{theorem}
We remark that Hochman stated only two of the four inequalities required in the above result, but the remaining two follow with the same proof.

With the above properties of fractal distributions and uniformly scaling measures in mind, we can now prove a previously unrecorded property of fractal distributions that will allow us to ignore measures which give positive mass to boundaries of balls. This will be crucial when we deal with the porosity results, as it will allows us to pass between open and closed balls with ease.

\begin{theorem}\label{thm:nomass}
If $P$ is a fractal distribution with $\dim P > 0$, then for $P$ almost every $\nu$ we have $\nu(\partial B) = 0$ for all balls $B$.
\end{theorem}

\begin{proof}
Suppose to the contrary that the set
\begin{equation*}
  \cA := \{ \nu \in \cM_1 : \nu(\partial B) > 0 \text{ for some ball } B \}
\end{equation*}
has positive measure, $P(\cA)>0$. By the ergodic decomposition, we may assume that $P$ is ergodic. According to Theorem \ref{thm:ergodicgenerates}, $P$ almost every $\nu \in \cA$ is a uniformly scaling measure for $P$. Recalling that for each $\nu \in \cA$ there is a ball $B$ so that $\nu(\partial B)>0$, Theorem \ref{thm:restriction} shows that for $P$ almost every $\nu$ the normalized restriction $\nu_{\partial B}$ is a uniformly scaling measure for $P$. Each such $\nu_{\partial B}$ is supported on a $(d-1)$-dimensional sphere $\partial B$ and hence $P$ almost every $\nu \in \cA$ is supported on a $(d-1)$-dimensional plane. This is because tangent measures of measures supported on $\partial B$ are supported on a $(d-1)$-dimensional plane. Thus, in particular, $P$ almost every measure $\nu \in \cA$ is supported on a $(d-1)$-dimensional plane. Since $\nu$ is supported on a $(d-1)$-plane $V_1$ and $\nu(\partial B_1)>0$, we have $\nu(\partial B_1\cap V_1)>0$, where the intersection $\partial B_1 \cap V_1$ is either $(d-2)$-dimensional or a single point. If the intersection is one point, then $P = \delta_{\delta_0}$ which is a contradiction by Lemma \ref{lem:onlydim}.


Now we continue inductively and show that $P$ almost every $\nu$ gives positive measure for a $(d-3)$-dimensional set $\partial B \cap V_2$ where $V_2 \in G(d,d-2)$. Eventually, we are at dimension $1$ in which case, since the intersection of a line and $\partial B \cap V_{d-2}$ where $V_{d-2} \in G(d,2)$ is at most two points, the conclusion is that $P = \delta_{\delta_0}$. This contradiction finishes the proof.
\end{proof}

Since a fractal distribution cannot give positive mass to measures that charge the boundary $\partial B_1$ we get a more quantitative version of the quasi-Palm property:

\begin{lemma}
\label{lem:nomassquasipalm}
Suppose $P$ is a fractal distribution. Then for any Borel set $\cA \subset \cM_1$ with $P(\cA) = 1$ we have that $P$ almost every $\nu \in \cA$ and for $\nu$ almost every $z \in \R^d$ there exists $t_z > 0$ such that for $t \geq t_z$ we have $B(z,e^{-t}) \subset B_1$ and $\nu_{z,t}  \in \cA$.
\end{lemma}

\begin{proof}
Since, by Theorem \ref{thm:nomass}, $P$ almost every $\nu$ and $\nu$ almost every $z \in \R^d$ satisfy $z \notin \partial B_1$ the claim follows immediately from the definition of the quasi-Palm property.
\end{proof}

Recall that, though the action $S_t$ is discontinuous, it does share many good properties with continuous flows on compact spaces. Another manifestation of this principle is the following lemma which asserts that $S_t$ cannot escape from compact closed collections of measures:

\begin{lemma} \label{lem:escapeclosed}
If $P$ is a fractal distribution and $\cA \subset \cM_1$ is closed with $P(\cA) = 1$, then $\nu_{0,t} \in \cA$ for $P$ almost every $\nu$ and for all $t \ge 0$.
\end{lemma}

\begin{proof}
Suppose the claim does not hold. Then in a set $\cA_0 \subset \cA$ with $P(\cA_0) > 0$ we find $t_0 \in \R$ such that for each $\nu \in \cA_0$ we have $\nu_{0,t_0} \notin \cA$.  Since $\cA$ is closed in the compact metric space $\cM_1$, there is an open neighbourhood $\cU$ of $\nu_{0,t_0}$ with no members in $\cA$. By Theorem \ref{thm:nomass}, we may assume $\nu$ gives no mass to $(d-1)$-dimensional spheres. In particular, we have
$$\nu_{0,t} \to \nu_{0,t_0}$$
as $t \to t_0$. Thus there is an interval $I$ containing $t_0$ such that $\nu_{0,t} \notin \cA$ for any $t \in I$. On the other hand, since $P$ is scale invariant, the countable intersection
$$\widehat{\cA} = \bigcap_{t \in \Q_+} S_t^{-1} \cA$$
has full $P$ measure. Since $P(\cA_0) > 0$ we can choose $\nu \in \cA_0 \cap \widehat{\cA}$. Then for all rational $t \in I$, we have $\nu_{0,t} \in \cA$, which is a contradiction.
\end{proof}

\section{Proofs of the results} \label{sec:proofs}

\subsection{General strategy behind the proofs}
Although the proofs of the main results differ, there is a general outline common to all of them. The reader may want to keep these steps in mind while going through the proofs.

\begin{itemize}
 \item[(1)] We begin from a geometric property of measures, which is described by some threshold $\epsilon > 0$ (such as $\eps$-mass in cones or $\eps$-porosity), and form a collection of measures $\cA_\eps \subset \cM_1$ that describes the geometric property under study.
 \item[(2)] Derive information from the geometric property to obtain bounds on the frequency with which the orbit $(\mu_{x,t})_{t \geq 0}$ hits the set $\cA_\eps$; recall Remark \ref{rem:scenery_geom}. Then weak limits of the scenery $\la \mu \ra_{x,T}$, that is, tangent distributions $P_\eps$, will give mass to $\cA_\eps$ according to this frequency.
 \item[(3)] Invoke the fact that the limiting distribution $P_\eps$ is a fractal distribution at a typical $x$ (Theorem \ref{thm:local-tangent}), and after possibly passing to ergodic components, allow $\eps \to 0$ to obtain a limit set of measures $\cA$ from $(\cA_\eps)_{\eps > 0}$, which supports a weak accumulation point $P$ of the distributions $P_\eps$. Furthermore, $P_\eps$ can be chosen to satisfy certain geometric conditions (such as a dimension bound) that pass to the limit, so that $P$ satisfies the same conditions.
 \item[(4)] By the compactness of fractal distributions (Theorem \ref{thm:Fds-are-closed}), this distribution $P$ is still a fractal distribution. Moreover, the set $\cA$ contains measures $\nu$ for which we know geometric information about their supports around the origin and at scale $1$.  By the quasi-Palm property of $P$, this information extends to $\nu$ almost every other point $x$ and all small enough scales $e^{-t}$.
 \item[(5)] As the final step, we need to pass back from fractal distributions to sets and measures. This can be done either by showing that $P$ almost all $\nu$ (or their supports) satisfy the required conditions, or by showing that a uniformly scaling measure generating $P$ does (recall from Theorem \ref{thm:measure-generating-FD} that such a uniformly scaling measure always exists).
\end{itemize}

\subsection{Conical densities} \label{sec:conical-proofs}
In this section, we prove Theorem \ref{thm:conical} which shows that if the dimension of the measure is large, then there are many scales in which the non-symmetric cones contain a large portion of the mass from the surrounding ball. We begin the proof by slightly improving the rectifiability criterion given in \cite[Lemma 15.13]{Mattila1995}. The following lemma can be considered to be the set theoretical version of the conical density theorem (in contrareciprocal form).

\begin{lemma}[Rectifiability criterion] \label{thm:rect_lemma}
  A set $E \subset \R^d$ is strongly $k$-rectifiable if for every $x \in E$ there are $V \in G(d,d-k)$, $\theta \in S^{d-1}$, $0 < \alpha < 1$, and $r > 0$ so that
  $$
    E \cap X(x,r,V,\alpha) \setminus H(x,\theta,\alpha) = \emptyset.
  $$
\end{lemma}

\begin{proof}
  Expressing $E$ suitably as a countable union, we can assume that $V$, $\theta$, $\alpha$, and $r$ do not depend on $x$. The orthogonal projection onto $V$ is denoted by $\proj_V$ and the orthogonal complement of $V$ is $V^\bot$.  To apply the argument used in the proof of \cite[Lemma 15.13]{Mattila1995}, one has to notice that if $x,y \in E$ so that $|y-x| < r$ and $|\proj_{V^\bot}(y - x)| < \alpha|y-x|$, then not only $y \in X(x,r,V,\alpha) \cap H(x,\theta,\alpha)$ but also $x \in X(y,r,V,\alpha) \setminus H(y,\theta,\alpha)$. This observation guarantees the existence of a Lipschitz mapping between $\proj_{V^\bot}(E)$ and $E$, so $E$ is strongly $k$-rectifiable.
\end{proof}

The next lemma shows that the property of having small mass in a certain cone is a closed property in the space of measures. This is a necessary technical tool in the proof of Theorem \ref{thm:conical}. For this purpose, we fix $0 < \alpha \le 1$ and, for a parameter $\eps\ge 0$, write
\[
\AA_\eps := \{ \nu \in \cM_1 : \nu(X(0,1,V,\alpha) \setminus H(0,\theta,\alpha)) \le \eps \text{ for some $V \in G(d,d-k)$ and $\theta \in S^{d-1}$} \}.
\]

\begin{lemma} \label{lem:conical_closed}
The set $\cA_\eps$ closed in $\cM_1$ for all $\eps \ge 0$.
\end{lemma}

\begin{proof}
  Suppose that there is a sequence $(\nu_i)_i$ of measures in $\AA_\eps$ and $\nu \in \cM_1$ so that $\nu_i \to \nu$ weakly. Let $V_i \in G(d,d-k)$ and $\theta_i \in S^{d-1}$ be so that
  \begin{equation*}
    \nu_i(X(0,1,V_i,\alpha) \setminus H(0,\theta_i,\alpha)) \le \eps.
  \end{equation*}
  The compactness of $G(d,d-k)$ and $S^{d-1}$ allows us to extract $V \in G(d,d-k)$ and $\theta \in S^{d-1}$ such that, possibly passing to a subsequence, $V_i \to V$ and $\theta_i \to \theta$. Now for each $0<\eta<1$ we have $i_\eta$ so that
  \begin{equation*}
    C(\eta) := X(0,1,V,\eta\alpha) \setminus H(0,\theta,\alpha/\eta) \subset X(0,1,V_i,\alpha) \setminus H(0,\theta_i,\alpha)
  \end{equation*}
  for all $i \ge i_\eta$. Since the cones $C(\eta)$ are open we have
  \[
  \nu(C(\eta)) \le \liminf_{i \to \infty} \nu_i(C(\eta)) \le \eps
  \]
  for all $0<\eta<1$. Moreover, since $C(\eta_1) \subset C(\eta_2)$ for all $\eta_1 \le \eta_2$ and $X(0,1,V,\alpha) \setminus H(0,\theta,\alpha) = \bigcup_{0<\eta<1} C(\eta)$, we also have $\nu(X(0,1,V,\alpha) \setminus H(0,\theta,\alpha)) \le \eps$ and thus $\nu \in \AA_\eps$.
\end{proof}

We are now ready to prove Theorem \ref{thm:conical} by following the outline above together with the rectifiability criterion of Lemma \ref{thm:rect_lemma}. We split the proof into the two propositions below.

\begin{proposition} \label{thm:conicalLEMMA1}
If $d \in \N$, $k \in \{ 1,\ldots,d-1 \}$, $k<s\le d$, and $0<\alpha\le 1$, then there exists $0 < \eps < \eps(d,k,\alpha)$ satisfying the following: For every Radon measure $\mu$ on $\R^d$ with $\ldimh \mu \ge s$ it holds that
\begin{equation}\label{eq:lowerconicalporousLEMMA}
  \liminf_{T \to \infty} \frac{1}{T} \,\lambda\Big(\Big\{ t \in [0,T] : \inf_{\yli{\theta \in S^{d-1}}{V \in G(d,d-k)}} \frac{\mu(X(x,e^{-t},V,\alpha) \setminus H(x,\theta,\alpha))}{ \mu(B(x,e^{-t}))} > \epsilon\Big\}\Big) \ge \frac{s-k}{d-k}
\end{equation}
at $\mu$ almost every $x \in \R^d$. If the measure $\mu$ only satisfies $\ldimp\mu \ge s$, then \eqref{eq:lowerconicalporousLEMMA} holds with $\limsup_{T \to \infty}$ at $\mu$ almost every $x \in \R^d$.
\end{proposition}

\begin{proof}
Let $0<p<(s-k)/(d-k)$. Suppose to the contrary that there is $0<\alpha\le 1$ so that for each $0<\eps<\eps(d,k,\alpha)$ there exists a Radon measure $\mu$ with $\ldimh\mu \ge s$ such that the condition \eqref{eq:lowerconicalporousLEMMA} fails to hold for $p$, that is,
\[
\limsup_{T \to \infty}\,\la \mu \ra_{x,T}(\cA_\epsilon) > 1-p,
\]
on a set $E_\eps$ of positive $\mu$ measure, where $\AA_\eps$ is as in Lemma \ref{lem:conical_closed} (recall Remark \ref{rem:scenery_geom}).

Fix $\delta>0$ such that $p < (s-\delta-k)/(d-k)$.
Recalling Theorems \ref{thm:TDs-are-FDs} and \ref{thm:local-tangent}, we may assume that all tangent distributions of $\mu$ at points $x \in E_{\epsilon}$ are fractal distributions, and
\[
\llocd(\mu,x) = \inf\{\dim P : P \in \TD(\mu,x)\} > s-\delta.
\]
Fix $x \in E_{\epsilon}$. For each $0<\epsilon<\eps(d,k,\alpha)$, as $\cA_\eps$ is closed by Lemma \ref{lem:conical_closed}, we find a tangent distribution $P_{\eps} \in \TD(\mu,x)$ so that $P_{\eps}(\cA_\epsilon) \ge 1-p$. If $P$ is a weak limit of a sequence formed from $P_\epsilon$ as $\epsilon \downarrow 0$, then, since the sets $\cA_\epsilon$ are nested and closed, we have
\[
P(\cA_\eps) \ge \limsup_{\eta \downarrow 0} P_{\eta}(\cA_\eps) \ge \limsup_{\eta\downarrow 0} P_{\eta}(\cA_{\eta}) \ge 1-p,
\]
and thus
\[
P(\cA_0) = \lim_{\eps \downarrow 0} P(\cA_\eps) \ge 1-p.
\]
Furthermore, since, by Theorem \ref{thm:Fds-are-closed} and Lemma \ref{lem:continuity-of-FD-dim}, the collection of all fractal distributions is closed and the dimension is continuous, the limit distribution $P$ is a fractal distribution with $\dim P \geq s-\delta$.

A key observation is that $\cA_0$ is $S$-invariant (up to $P$-measure zero). Indeed, it follows from the definitions that $\cA_0\subset S_t^{-1} \cA_0$ for all $t \ge 0$. Since $P$ is $S_t$ invariant, that is, $P(\cA_0)= P(S_t^{-1} \cA_0)$ for all $t \ge 0$, we see that the set $\cA_0$ is $S_t$ invariant up to $P$-measure zero. Let
\[
P = \int P_\omega \, \mathrm{d} P(\omega)
\]
be the ergodic decomposition of $P$. By the invariance of $\cA_0$, we have $P_\omega(\cA_0) \in \{0,1\}$ for $P$ almost all $\omega$. If $P_\omega(\cA_0) = 0$, we use the trivial estimate $\dim P_\omega \le d$.  If $P_\omega(\cA_0) = 1$, then, using the quasi-Palm property in the form of Lemma \ref{lem:nomassquasipalm}, for $P_\omega$ almost every $\nu$ and for $\nu$ almost every $z$ the normalized translation $\nu_{z,t_z}$ is an element of $\cA_0$ for some $t_z > 0$ with $B(z,e^{-t_z}) \subset B_1$.
For each such $\nu$ let $E = \{ z \in B_1 : \nu_{z,t_z} \in \cA_0 \}$ be this set of full $\nu$ measure. Thus for every $z \in E$ there are $V \in G(d,d-k)$ and $\theta \in S^{d-1}$ with
$$
  E \cap X(z,e^{-t_z},V,\alpha) \setminus H(z,\theta,\alpha) = \emptyset.
$$
Lemma \ref{thm:rect_lemma} implies that $E$ is strongly $k$-rectifiable. In particular, $\dim \nu \le k$, which yields $\dim P_\omega \le k$.

Since $P_\omega(\AA_0) \in \{0,1\}$ for $P$ almost all $\omega$, we have
\[
1-p \le P(\AA_0) = \int P_\omega(\AA_0) \, \mathrm{d}P(\omega) = \int_{\{\omega : P_\omega(\AA_0) = 1\}} P_\omega(\AA_0) \, \mathrm{d}P(\omega) = P(\{\omega : P_\omega(\AA_0) = 1\}).
\]
Using this, we estimate
$$
  s-\delta \leq \dim P = \int \dim P_\omega \,\mathrm{d}P(\omega) \leq P(\cA_0)k + (1-P(\cA_0))d \le (1-p)k + pd
$$
which gives $p \ge (s-\delta-k)/(d-k)$. But this contradicts the choice of $\delta$. Thus the claim holds.

The proof of the second claim is almost identical: We choose a distribution $P$ so that $\dim P$ is close to $\ldimp \mu$. Since \eqref{eq:lowerconicalporous} fails with $\limsup_{T \to \infty}$, we know that this distribution gives large measure to $\AA_\eps$. Continuing as above finishes the proof of the second claim.
\end{proof}

It remains to show the sharpness of Theorem \ref{thm:conical}:

\begin{proposition} \label{thm:conicalLEMMA2}
Let $d \in \N$, $k \in \{ 1,\ldots,d-1 \}$, $k<s \leq d$, and $0<\alpha\le 1$. Then there exists a Radon measure $\mu$ of exact dimension $s$ such that \eqref{eq:lowerconicalporousLEMMA} holds with $\lim_{T \to \infty}$ and the limit equals $(s-k)/(d-k)$ for all $0<\eps<\eps(d,k,\alpha)$, but is equal to zero for all $\eps > \eps(d,k,\alpha)$.
\end{proposition}

\begin{proof}
 Fix $W \in G(d,k)$ and let
\begin{equation*}
  P = \frac{s-k}{d-k}\delta_{\overline{\LL}^d} + \Big(1 - \frac{s-k}{d-k}\Big)\delta_{\overline{\cH}_W},
\end{equation*}
where  $\overline{\cH}_W$ is the normalization of $\HH^k|_{W \cap B_1}$ and $\ol{\cL}^d$ is the normalization of $\cL^d|_{B_1}$. Since $P$ is a convex combination of two fractal distributions, it is a fractal distribution. Recalling Theorem \ref{thm:measure-generating-FD}, we let $\mu$ be a uniformly scaling measure generating $P$. Theorem \ref{thm:local-tangent} shows that $\mu$ is exact dimensional and
\begin{equation*}
  \dim\mu = \dim P = \frac{s-k}{d-k}\,d + \Big(1 - \frac{s-k}{d-k}\Big)k = s.
\end{equation*}
Our goal is to verify that $\mu$ has the claimed properties.

Recalling \eqref{eq:def-eps} fix $0<\epsilon < \eps(d,k,\alpha)$. Notice that, as the set $\cA_\eps$ of Lemma \ref{lem:conical_closed} is closed, the complement
$$
  \cA_\eps^c = \{ \nu \in \cM_1 : \nu(X(0,1,V,\alpha) \setminus H(0,\theta,\alpha)) > \eps \text{ for all $V \in G(d,d-k)$ and $\theta \in S^{d-1}$} \}
$$
is open. Moreover, $P(\cA_\eps^c) = (s-k)/(d-k)$ since
\[
\ol{\LL}^d(X(0,1,V,\alpha) \setminus H(0,\theta,\alpha)) \geq \eps(d,k,\alpha) > \eps
\]
for all $V \in G(d,d-k)$ and $\theta \in S^{d-1}$, and this does not hold for $\ol{\HH}_W$. Thus by the weak convergence
\begin{equation} \label{eq:conical_liminf}
  \liminf_{T \to \infty}\, \la \mu \ra_{x,T}(\cA_\eps^c) \geq P(\cA_\eps^c) = \frac{s-k}{d-k}.
\end{equation}
Moreover, if $\nu_i \in \cA_\eps^c$ with $\nu_i \to \nu$, then for any $V \in G(d,d-k)$ and $\theta \in S^{d-1}$ the weak convergence also gives
$$
  \nu(\overline{X(0,1,V,\alpha) \setminus H(0,\theta,\alpha)}) \geq \limsup_{i \to \infty} \nu_i(\overline{X(0,1,V,\alpha) \setminus H(0,\theta,\alpha)}) \geq \eps.
$$
Therefore, the closure of the set $\cA_\eps^c$ is
\begin{equation} \label{eq:closure_of_A}
  \overline{\cA_\eps^c} =  \{ \nu \in \cM_1 : \nu(\overline{X(0,1,V,\alpha) \setminus H(0,\theta,\alpha)}) \ge \eps \text{ for all $V \in G(d,d-k)$ and $\theta \in S^{d-1}$} \}.
\end{equation}
This implies
\begin{equation} \label{eq:conical_limsup}
  \limsup_{T \to \infty}\, \la \mu \ra_{x,T}(\cA_\eps^c) \leq \limsup_{T \to \infty} \,\la \mu \ra_{x,T}(\overline{\cA_\eps^c}) \leq P(\overline{\cA_\eps^c}) = \frac{s-k}{d-k},
\end{equation}
since $\ol{\LL}^d(\overline{X(0,1,V,\alpha) \setminus H(0,\theta,\alpha)}) \geq \eps(d,k,\alpha) > \eps$ for all $V \in G(d,d-k)$ and $\theta \in S^{d-1}$, and this does not hold for $\ol{\HH}_W$. Now \eqref{eq:conical_liminf} and \eqref{eq:conical_limsup} together show that \eqref{eq:lowerconicalporous} holds with $\lim_{T \to \infty}$ and the limit equals $(s-k)/(d-k)$.

To finish the proof, we are left to consider the case $\eps > \eps(d,k,\alpha)$. The claim follows almost immediately from the choices of $\eps(d,k,\alpha)$ and the measure $\mu$. Indeed,
$$
  \limsup_{T \to \infty}\, \la \mu \ra_{x,T}(\cA_\eps^c) \le P(\overline{\cA_\eps^c})=0
$$
since there exist $V \in G(d,d-k)$ and $\theta \in S^{d-1}$ such that $\ol{\LL}^d(\overline{X(0,1,V,\alpha) \setminus H(0,\theta,\alpha)}) = \eps(d,k,\alpha)$.
\end{proof}

\subsection{Average unrectifiability} \label{sec:average}
In this section we prove Theorem \ref{thm:unrectifiable}, which shows that average unrectifiability is also a sufficient condition to guarantee that the measure is scattered enough in the sense of conical densities, and Theorem \ref{thm:unrectifiable-converse}, which provides a converse under an additional assumption on the $k$-dimensional densities of $\mu$.

\begin{proof}[Proof of Theorem \ref{thm:unrectifiable}]
We begin the proof by showing that it suffices to prove the statement on a set of positive $\mu$ measure. Let $A$ be the set of points $x$ where the statement fails. If $A$ does not have zero $\mu$ measure, let $\nu = \mu_A$. By Theorem \ref{thm:restriction}, the hypothesis holds also for $\nu$, so there is a set $B\subset A$ of positive $\nu$ measure (so also of positive $\mu$ measure) where the statement holds for $\nu$. But this is a contradiction since for $x\in B$,
\[
  \frac{\mu(X(x,e^{-t},V,\alpha) \setminus H(x,\theta,\alpha))}{\mu(B(x,e^{-t}))} \ge  \frac{\mu|_A(X(x,e^{-t},V,\alpha) \setminus H(x,\theta,\alpha))}{\mu|_A(B(x,e^{-t}))} \frac{\mu|_A(B(x,e^{-t}))}{\mu(B(x,e^{-t}))}.
\]
whence, thanks to the Besicovitch density point theorem \cite[Corollary 2.14]{Mattila1995}, the statement holds also for $\mu$ almost all $x \in B$.

Now suppose to the contrary that a $p$-average $k$-unrectifiable measure $\mu$ and $0<\alpha\le 1$ are such that for each $0<\eps\le 1$ the condition \eqref{eq:conicalporous} fails to hold in a set $E_\eps$ of full $\mu$ measure. Recalling Theorem \ref{thm:TDs-are-FDs}, we may assume that all tangent distributions at points $x \in E_\eps$ are fractal distributions and satisfy $P(\{ \nu \in \MM_1 : \spt \nu \text{ is strongly $k$-rectifiable} \}) < 1-p$.
Let $x \in \bigcap E_\eps$, where the intersection is over all rational $0<\eps\le 1$. Then \eqref{eq:conicalporous} fails at $x$ and
\begin{equation*}
\limsup_{T \to \infty}\, \la \mu \ra_{x,T}(\cA_\eps) > 1-p
\end{equation*}
for all rational $0<\eps\le 1$, where $\cA_\eps$ is the closed set from Lemma \ref{lem:conical_closed}.  We choose a tangent distribution $P$ so that $P(\cA_\eps) \ge 1-p$ for all rational $0<\eps\le 1$. Since the sets $\cA_\eps$ are nested and closed, we get $P(\cA_0) \ge 1-p$.

Let
\[
P = \int P_\omega \, \mathrm{d}P(\omega)
\]
be the ergodic decomposition of $P$. As remarked in the proof of Theorem \ref{thm:conical}, $\cA_0$ is $S_t$ invariant up to $P$ measure zero. Thus we have $P_\omega(\cA_0) \in \{0,1\}$. Notice that
\[
P(\{ \omega : P_\omega(\cA_0) = 1 \}) = P(\cA_0) \ge 1-p.
\]
If $P_\omega(\cA_0) = 1$, then, by the quasi-Palm property of Lemma \ref{lem:nomassquasipalm}, for $P_\omega$ almost every $\nu$ and for $\nu$ almost every $z$ the normalized translation $\nu_{z,t_z}$ is an element of $\cA_0$ for some $t_z > 0$ with $B(z,e^{-t_z}) \subset B_1$. For each such $\nu$ let $E = \{ z \in B_1 : \nu_{z,t_z} \in \cA_0 \}$ be this set of full $\nu$ measure. Thus for every $z \in E$ there are $V \in G(d,d-k)$ and $\theta \in S^{d-1}$ with
$$
  E \cap X(z,e^{-t_z},V,\alpha) \setminus H(z,\theta,\alpha) = \emptyset.
$$
Lemma \ref{thm:rect_lemma} implies that $E$ is strongly $k$-rectifiable. Consequently,
\[
P_\omega(\{ \nu \in \cM_1 : \spt\nu \text{ is strongly $k$-rectifiable} \}) = 1.
\]
Thus by the ergodic decomposition
$$
  P(\{ \nu \in \cM_1 : \spt\nu \text{ is strongly $k$-rectifiable} \}) \geq 1-p.
$$
The proof of the claim is now finished since this contradicts the $p$-average $k$-unrectifiability assumption.
\end{proof}

Our next goal is to prove Theorem \ref{thm:unrectifiable-converse}. Before doing so, we state a lemma that will be required in its proof.

\begin{lemma} \label{lem:tangents-of-rectifiable}
Let $P$ be an ergodic fractal distribution such that
\[
P(\{\nu\in\cM_1: \nu \text{ is } k\text{-rectifiable}  \}) > 0.
\]
Then there exists $V\in G(d,k)$ such that $P=\delta_{\overline{\cH}_V}$, where $\overline{\cH}_V$ is the normalized restriction of $\cH^k$ to $V\cap B_1$.
\end{lemma}

\begin{proof}
By Theorem \ref{thm:ergodicgenerates}, there exists a uniformly scaling measure $\nu$ which is $k$-rectifiable and generates $P$. Hence there exists a strongly $k$-rectifiable set $E$ such that $\nu(\R^d\setminus E)=0$. Write $E=\bigcup_j E_j$ where each $E_j$ is a Lipschitz graph of positive and finite $\cH^k$ measure. Setting $\eta_j=\cH^k|_{E_j}$, we have by \cite[Theorem 16.5]{Mattila1995} that for $\eta_j$ almost all $x$ there exists $V\in G(d,k)$ such that $\Tan(\eta_j,x)=\{ \overline{\cH}_V\}$, and therefore $\TD(\eta_j,x)=\{ \delta_{\overline{\cH}_V}\}$. We remark that although \cite[Theorem 16.5]{Mattila1995} has a density assumption, this is not needed for the implication (1)$\Rightarrow$(2) which is all we use; see \cite[Remark 16.8]{Mattila1995}. We also recall that, unlike the classical definition of tangent measures, we are restricting to the unit ball and renormalizing to get probability measures.

Now since $\nu|_{E_j}\ll\eta_j$ by assumption (as $\nu\ll\cH^k$ by $k$-rectifiability), thanks to Theorem \ref{thm:restriction} we also have that for $\nu|_{E_j}$ almost all $x$ there is $V\in G(d,k)$ such that $\TD(\nu,x)=\{ \delta_{\overline{\cH}_V}\}$. As $E=\bigcup_j E_j$, the same conclusion holds for $\nu$. But since $\nu$ is uniformly scaling, $V$ must be independent of $x$, and we are done.
\end{proof}

\begin{proof}[Proof of Theorem \ref{thm:unrectifiable-converse}]
Let $\eps,\alpha,p$ be as in the statement. Fix a point $x$ so that both assumptions are satisfied at $x$ and all tangent distributions at $x$ are fractal distributions; recall Theorem \ref{thm:TDs-are-FDs}. Let
\begin{equation} \label{eq:large-mass-in-cones}
  \cB =  \{ \nu \in \cM_1 : \nu(\overline{X(0,1,V,\alpha)}) \ge \eps \text{ for all $V \in G(d,d-k)$} \}.
\end{equation}
As before, since we deal with closed cones, the set $\cB$ is closed. Using this and the hypothesis \eqref{eq:conicalporous-converse}, we have
\begin{equation*}
  P(\cB) \ge \liminf_{T \to \infty} \,\la \mu \ra_{x,T}(\cB) > p
\end{equation*}
for all $P \in \TD(\mu,x)$. For a given $P\in\TD(\mu,x)$, consider its ergodic decomposition
\begin{equation*}
  P = \int P_\omega \,\mathrm{d}P(\omega).
\end{equation*}
It follows from Lemma \ref{lem:tangents-of-rectifiable} and the assumption on $\mu$ (recall Remark \ref{rmk:unrectconverse}) that if
\[
P_\omega(\{\nu\in\cM_1:\spt\nu \text{ is $k$-rectifiable}\}) > 0,
\]
then $P_\omega(\cB)=0$. Hence
\[
p <  P(\cB) = \int P_\omega(\cB) \, \mathrm{d}P(\omega) \le P(\{\nu\in\cM_1:\spt\nu \text{ is not $k$-rectifiable}\}).
\]
As $x$ was a $\mu$ typical point and $P\in\TD(\mu,x)$ was arbitrary, this gives the claim.
\end{proof}

To conclude the discussion on conical densities, we give two relevant examples of average unrectifiable measures. The next lemma links Theorems \ref{thm:conical} and \ref{thm:unrectifiable}; recall Remark \ref{rem:linkThmsAB}(1).

\begin{lemma} \label{lem:dim-unrectifiability}
If $\mu$ is a Radon measure on $\R^d$ such that $\ldimh(\mu) > s > k$ for some $k\in\{1,\ldots,d-1\}$, then $\mu$ is $\tfrac{s-k}{d-k}$-average $k$-unrectifiable.
\end{lemma}

\begin{proof}
By Theorems \ref{thm:TDs-are-FDs} and \ref{thm:local-tangent}, at $\mu$ almost every $x \in \R^d$ all elements of $\TD(\mu,x)$ are fractal distributions and
\begin{equation*}
  s < \ldimloc(\mu,x) = \inf\{ \dim P : P \in \TD(\mu,x) \}.
\end{equation*}
 Pick such a point $x$, choose any $P \in \TD(\mu,x)$, and write
\begin{equation*}
  \cA = \{ \nu \in \cM_1 : \spt\nu \text{ is not strongly $k$-rectifiable} \}.
\end{equation*}
It is clear that if $\spt\nu$ is strongly $k$-rectifiable, then $\udimp \nu \le k$, and that $\udimp\nu\le d$ holds for any measure $\nu$ on $\R^d$. Recalling Definition \ref{def:dimension}, we deduce that
\begin{equation*}
  s < \dim P  =\int \dim\nu \,\mathrm{d}P(\nu) \le P(\cA)d + (1-P(\cA))k = k + (d-k)P(\cA).
\end{equation*}
Hence $P(\cA) > \frac{s-k}{d-k}$, showing that $\mu$ is $\frac{s-k}{d-k}$-average $k$-unrectifiable, as claimed.
\end{proof}

In the following lemma, we prove the existence of an average unrectifiable measure for a given proportion $p$; recall Example \ref{ex:unrect}(3).

\begin{lemma}
\label{lma:constructionofunrect}
Given $0 \leq p \leq 1$, there exists a uniformly scaling measure $\mu$ generating a fractal distribution $P$ with
$$P(\{\nu \in \cM_1 : \spt\nu \text{ is not strongly $k$-rectifiable}\}) = p.$$
\end{lemma}

\begin{proof}
The proof is similar to what we already did in Proposition \ref{thm:conicalLEMMA2}. Let $E \subset \R^d$ be a self-similar set with a strong separation condition and of dimension $k$. By the self-similarity of $E$, the Hausdorff measure $\cH^k|_E$ is uniformly scaling generating an ergodic fractal distribution $Q$ supported on measures $\nu$ such that $\spt \nu$ is a translated and scaled copy of $E$ restricted to the unit ball; see Bandt \cite{Bandt2001}. Thus for $Q$ typical $\nu$ the support $\spt \nu$ is also purely $k$-unrectifiable. This yields that
$$Q(\{\nu \in \cM_1 : \spt\nu \text{ is not strongly $k$-rectifiable}\}) = 1.$$
Now fixing $V \in G(d,k)$ and defining
\begin{equation*}
  P = pQ + (1-p)\delta_{\overline{\cH}_V},
\end{equation*}
where again $\overline{\cH}_V$ is the normalization of $\HH^k|_{V \cap B_1}$, we obtain a fractal distribution $P$ that satisfies
$$P(\{\nu \in \cM_1 : \spt\nu \text{ is not strongly $k$-rectifiable}\}) = p.$$
By Theorem \ref{thm:measure-generating-FD}, we find a uniformly scaling measure $\mu$ generating $P$, so the proof is complete.
\end{proof}

\subsection{Porosity}
In this section, we prove Theorem \ref{thm:mainporosity} which shows that the maximal dimensions of porous sets and measures, whether considering Hausdorff or packing dimension, are the same.
Let $0 < \alpha \le \frac{1}{2}$ be fixed and, for a parameter $\eps\ge 0$, write
\[
\AA_\eps := \{ \nu \in \cM_1 : \nu(B(y,\alpha)) \leq \epsilon \text{ for some } y \in \overline{B}(0,1-\alpha) \}.
\]

\begin{lemma} \label{lem:closedcollection}
The set $\cA_\eps$ is closed in $\cM_1$ for all $\eps \ge 0$.
\end{lemma}

\begin{proof}
Suppose that there are a sequence $(\nu_i)_i$ of measures in $\AA_\eps$ and $\nu \in \cM_1$ so that $\nu_i \to \nu$ weakly. Then for each $i$ there is $y_i$ so that $y_i \in  \overline{B}(0,1-\alpha)$ and $\nu_i(B(y_i,\alpha)) \leq \epsilon$. By compactness, after possibly passing to a subsequence, we find $y \in \overline{B}(0,1-\alpha)$ such that $y_i \to y$. Let $\alpha' < \alpha$. Then for $i$ large enough, we have $B(y,\alpha') \subset B(y_i,\alpha)$, so $\nu_i(B(y,\alpha')) \leq \nu_i(B(y_i,\alpha)) \leq \epsilon$. Since $B(y,\alpha')$ is open, we have
\[
\nu(B(y,\alpha')) \le \liminf_{i \to \infty} \nu_i(B(y,\alpha')) \le \epsilon.
\]
Since $B(y,\alpha)$ is union of $B(y,\alpha')$ over all $\alpha' < \alpha$, we also have $\nu(B(y,\alpha)) \leq \eps$. This means that $\nu \in \cA_\epsilon$, showing that $\cA_\epsilon$ is closed.
\end{proof}

\begin{proof}[Proof of Theorem \ref{thm:mainporosity}]
Since any measure supported on an $\alpha$-porous set is $\alpha$-porous as there is nothing in the pores, we only need to show that $\Delta(\alpha) \le \sup\{ \dimh E : E \text{ is $\alpha$-porous} \}$, and the supremum in the definition of $\Delta(\alpha)$ is attained by some uniformly scaling measure.

To that end, fix $0 < \alpha \le \tfrac12$ and $\delta>0$, and take an $\alpha$-porous measure $\mu$ with $\udimp\mu > \Delta(\alpha)-\delta/3$. Further, pick a point $x$ such that $\mu$ is $\alpha$-porous at $x$, $\udimloc(\mu,x) > \udimp \mu -\delta/3$, and there exists a fractal distribution $P_\delta \in \TD(\mu,x)$ with
 \[
\ulocd(\mu,x) \leq \dim P_\delta+\delta/3.
 \]
This is possible by Theorem \ref{thm:local-tangent}. Note that our choices imply that $\dim P_\delta \geq \Delta(\alpha)-\delta$.

Fix $\eps>0$ and let $\cA_\eps$ be as in Lemma \ref{lem:closedcollection}.
Since $\cA_\epsilon$ is closed, $P_\delta$ is a tangent distribution, and $\mu$ is $\alpha$-porous at $x$ we have
\begin{equation*}
  P_\delta(\cA_\eps) \ge \liminf_{T \to \infty} \,\la \mu \ra_{x,T}(\cA_\eps) = 1
\end{equation*}
for all $\eps > 0$ (recall Remark \ref{rem:scenery_geom}). Now let $P$ be a limit of $P_\delta$ along some subsequence.  Then, by Theorem \ref{thm:Fds-are-closed} and Proposition \ref{prop:dimension-FDs}, the distribution $P$ is a fractal distribution and $\dim P\ge \Delta(\alpha)$. Since the sets $\cA_\eps$ are nested and closed, we get for every $\eta>0$ that
\[
 P(\cA_\eta) \ge \limsup_{\eps\to 0} P_\eps(\cA_\eta) \ge \limsup_{\eps\to 0} P_\eps(\cA_\eps) = 1,
\]
whence $P(\cA_0) = 1$.

Let
$$P = \int P_\omega \,\mathrm{d}P(\omega)$$
be the ergodic decomposition of $P$.
Since $P(\cA_0) = 1$ we have $P_\omega(\cA_0) = 1$ for almost every $\omega$. Moreover, as $\nu \mapsto \dim \nu$ is measurable, we have
$$\dim P = \int\dim \nu \,\mathrm{d}P(\nu) = \int\hspace{-2px}\int\dim \nu\, \mathrm{d}P_\omega(\nu)\,\mathrm{d}P(\omega).$$
Thus there exists $\omega$ so that $\dim P_\omega \ge \dim P$ and $P_\omega(\cA_0) = 1$. Since $P_\omega$ is a fractal distribution and $\cA_0$ is closed with $P_\omega(\cA_0)=1$, Lemma \ref{lem:escapeclosed} implies that $\nu_{0,t} \in \cA_0$ for $P_\omega$ almost every $\nu \in \cA_0$ and for all $t \ge 0$. Applying the quasi-Palm property of Lemma \ref{lem:nomassquasipalm} thus gives that $P_\omega$ almost every $\nu \in \cM_1$ and $\nu$ almost every $z \in B_1$ there exists $t_z > 0$ such that $B(z,e^{-t}) \subset B_1$ and $\nu_{z,t} \in \cA_0$ for all $t \geq t_z$. By the ergodicity of $P_\omega$, a $P_\omega$ typical $\nu$ satisfies $\dim \nu = \dim P_\omega$ so we can choose one such $\nu$ with $\nu(E) = 1$ for
$$
  E = \{ z \in B_1 : \text{there is } t_z > 0 \text{ such that } B(z,e^{-t}) \subset B_1 \text{ and } \nu_{z,t} \in \cA_0 \text{ for all } t \geq t_z \}.
$$
The set $E$ is $\alpha$-porous by definition and satisfies $\Delta(\alpha)\le \dim \nu \leq \dim E$. Thus we have equality throughout, and this completes the proof.
\end{proof}

\subsection{Upper semicontinuity}
In this section, we prove Theorem \ref{thm:semic} which shows that the function $\alpha \mapsto \Delta(\alpha)$ is upper semicontinuous. Since the function is decreasing, it suffices to show that it is left continuous. To emphasize the dependence on $\alpha$, let us denote the set $\cA_\eps$ of Lemma \ref{lem:closedcollection} by $\cA_\eps(\alpha)$.

\begin{proof}[Proof of Theorem \ref{thm:semic}]
Let $0 < \alpha \le \tfrac12$ and $(\alpha_n)_n$ be an increasing sequence so that $\lim_{n \to \infty} \alpha_n = \alpha$. For each $n$ let $\mu$ be $\alpha_n$-porous with $\udimp\mu > \Delta(\alpha_n) - 1/n$. Furthermore, pick a point $x$ such that $\mu$ is $\alpha_n$-porous at $x$, $\udimloc(\mu,x) > \udimp\mu - 1/n$, and there exists a fractal distribution $P_n \in \TD(\mu,x)$ with $\udimloc(\mu,x) \le \dim P_n + 1/n$.
This is possible by Theorem \ref{thm:local-tangent}. Thus we have $$\dim P_n \ge \Delta(\alpha_n) - 3/n$$ for all $n$.
Since $\cA_\eps(\alpha_n)$ is closed, $P_n$ is a tangent distribution, and $\mu$ is $\alpha_n$-porous at $x$ we have
\begin{equation*}
  P_n(\cA_\eps(\alpha_n)) \ge \liminf_{T \to \infty} \,\la \mu \ra_{x,T}(\cA_\eps(\alpha_n)) = 1.
\end{equation*}
for all $\eps > 0$ and $n$. Since the sets $\cA_\eps(\alpha_n)$ are closed and nested with respect to $\eps$, we get $P_n(\cA_0(\alpha_n)) = 1$ for all $n$. Recall that the sequence $(\alpha_n)_n$ is increasing. Hence also the sets $\cA_0(\alpha_n)$ are nested and closed. Thus, if $P_n \to P$ weakly, we have
\[
P(\cA_0(\alpha_n)) \ge \limsup_{m \to \infty} P_m(\cA_0(\alpha_n)) \ge \limsup_{m \to \infty} P_m(\cA_0(\alpha_m)) \ge 1
\]
and
\begin{equation*}
  P(\cA_0(\alpha)) = \lim_{n \to \infty} P(\cA_0(\alpha_n)) = 1.
\end{equation*}
Considering now the ergodic decomposition of $P$ and continuing as in the proof of Theorem \ref{thm:mainporosity}, we find an $\alpha$-porous exact-dimensional measure $\nu$ with $\dim P \le \dim \nu$. Thus $\dim P \le \Delta(\alpha)$. But since, by Lemma \ref{lem:continuity-of-FD-dim},
\begin{equation*}
  \dim P = \lim_{n \to \infty} \dim P_n \ge \lim_{n \to \infty} \Delta(\alpha_n)
\end{equation*}
we have shown that the function is left continuous.
\end{proof}

\subsection{Mean porosity}
In this section, we prove Theorem \ref{thm:meanporosity} which shows that the maximal Hausdorff and packing dimensions of mean porous measures are the same, and the function $p \mapsto \Delta(\alpha,p)$ is concave.
Let $0<\alpha\le\tfrac12$ be fixed and, for a parameter $\eps > 0$, write
$$
  \cU_\eps := \{ \nu \in \cM_1 : \nu(\overline{B}(y,\alpha)) < \epsilon \text{ for some } y \in \overline{B}(0,1-\alpha) \}.
$$
This set should be compared to the set $\cA_\eps$ of Lemma \ref{lem:closedcollection}. The use of closed ball and strict inequality guarantee that the set $\cU_\eps$ is open.

\begin{lemma} \label{lem:opencollection}
The set $\cU_\eps$ is open for all $\eps>0$.
\end{lemma}

\begin{proof}
Write
$$\cU_\epsilon = \bigcup_{y \in \overline{B}(0,1-\alpha)} \{\nu \in \cM_1 : \nu(\overline{B}(y,\alpha)) < \epsilon\}.$$
It suffices to show that each $\cV := \{\nu \in \cM_1 : \nu(\overline{B}(y,\alpha)) < \epsilon\}$ is open. If $\nu \in \MM_1$ and $\nu_i  \in \cM_1 \setminus \cV$ for all $i$ so that $\nu_i \to \nu$, then
$$\nu(\overline{B}(y,\alpha)) \geq \limsup_{i \to \infty} \nu_i(\overline{B}(y,\alpha)) \geq \eps.$$
Thus $\nu$ is in the complement of $\cV$ and so $\cV$ is open.
\end{proof}

\begin{proof}[Proof of Theorem \ref{thm:meanporosity}]
To prove the left-hand side equality, it suffices to show that
\[
\Delta(\alpha,p) \le \sup\{ \ldimh\nu : \nu\text{ is }p\text{-mean }\alpha\text{-porous}\},
\]
with the supremum attained by a uniformly scaling measure.

Fix $\delta>0$. Let $\mu$ be a $p$-mean $\alpha$-porous measure with $\udimp(\mu)>\Delta(\alpha,p)-\delta/3$. Pick a point $x$ such that $\mu$ is $p$-mean $\alpha$-porous at $x$, $\udimloc(\mu,x) > \udimp\mu - \delta/3$, and there exists a fractal distribution $P_\delta \in \TD(\mu,x)$ with $\udimloc(\mu,x) \le \dim P_\delta + \delta/3$. This is possible by Theorem \ref{thm:local-tangent}. Notice that $\dim P_\delta \ge \Delta(\alpha,p)-\delta$ by construction.

Let $\cA_\eps$ be as in Lemma \ref{lem:closedcollection}. Since $\mu$ is $p$-mean $\alpha$-porous at $x$ and $\cA_\eps$ is closed we have
\begin{equation*}
  P_\delta(\cA_\eps) \ge \liminf_{T \to \infty}\, \la \mu \ra_{x,T}(\cA_\eps) \ge p
\end{equation*}
for all $\eps > 0$. Let $P$ be a limit of $P_\delta$ along some subsequence. Then $P$ is a fractal distribution and $\dim P\ge \Delta(\alpha,p)$, using once again Theorem \ref{thm:Fds-are-closed} and Lemma \ref{lem:continuity-of-FD-dim}.
Since the set $\cA_\eps$ is closed we have $P(\cA_\eps) \ge p$ for all $\eps>0$.

According to Theorem \ref{thm:measure-generating-FD}, there exists a uniformly scaling measure $\nu$ that generates $P$. Theorem \ref{thm:local-tangent} says that $\nu$ is exact dimensional and $\dim\nu = \dim P\ge \Delta(\alpha,p)$. Thus the result follows if we manage to show that $\nu$ is $p$-mean $\alpha$-porous.

By Lemma \ref{lem:opencollection}, $\cU_\eps$ is open and we have
\begin{equation*}
  \liminf_{T \to \infty} \,\la \nu \ra_{z,T}(\cU_{2\eps}) \ge P(\cU_{2\eps})
\end{equation*}
for all $\eps>0$ at $\nu$ almost every $z$. Theorem \ref{thm:nomass} guarantees that $\nu(\overline{B}(y,\alpha)) = \nu(B(y,\alpha))$ for $P$ almost every $\nu$ and for all $y$. Therefore $P(\cU_{2\eps}) \ge P(\cA_\eps) \ge p$. Since $\cA_{2\eps} \supset \cU_{2\eps}$ it follows that
\begin{equation*}
  \liminf_{T \to \infty} \,\la \nu \ra_{z,T}(\cA_{2\eps}) \ge \liminf_{T \to \infty} \, \la \nu \ra_{z,T}(\cU_{2\eps}) \ge p
\end{equation*}
for all $\eps > 0$ at $\nu$ almost every $z$, that is, $\nu$ is $p$-mean $\alpha$-porous.

Let us then show the concavity of $p \mapsto \Delta(\alpha,p)$.
Fix $\delta > 0$ and let $\mu$ be an $\alpha$-porous measure so that $\udimp\mu \ge \Delta(\alpha) - \delta/3$. Pick a point $x$ such that $\mu$ is $\alpha$-porous at $x$, $\udimloc(\mu,x) > \udimp\mu - \delta/3$, and there exists a fractal distribution $P \in \TD(\mu,x)$ with $\udimloc(\mu,x) \le \dim P + \delta/3$. This is possible by Theorem \ref{thm:local-tangent}. Since $\mu$ is $\alpha$-porous at $x$ and $\cA_\eps$ is closed we have
\begin{equation*}
  P(\cA_\eps) \ge \liminf_{T \to \infty} \, \la \mu \ra_{x,T}(\cA_\eps) \ge 1
\end{equation*}
for all $\eps>0$. Let
\begin{equation*}
  Q = pP + (1-p)\delta_{\overline{\LL}^d}.
\end{equation*}
According to Theorem \ref{thm:measure-generating-FD}, there exists a uniformly scaling measure $\nu$ that generates $Q$. It follows from Definition \ref{def:dimension} and Theorem \ref{thm:local-tangent} that $\nu$ is exact dimensional with
\begin{equation*}
  \dim\nu = \dim Q = \int \dim\mu \, \mathrm{d}Q(\nu) = p\dim P + (1-p)d \ge p(\Delta(\alpha)-\delta) + (1-p)d.
\end{equation*}
Observe also that
\begin{equation*}
  Q(\cA_\eps) = pP(\cA_\eps) + (1-p)\delta_{\overline{\LL}^d}(\cA_\eps) \ge p
\end{equation*}
for all $\eps > 0$.
Recalling Lemma \ref{lem:opencollection} and Theorem \ref{thm:nomass}, we conclude that
\begin{equation*}
  \liminf_{T \to \infty} \, \la \nu \ra_{z,T}(\cA_{2\eps}) \ge \liminf_{T \to \infty} \,\la \nu \ra_{z,T}(\cU_{2\eps}) \ge Q(\cU_{2\eps}) \ge Q(\cA_\eps) \ge p
\end{equation*}
for all $\eps > 0$ at $\nu$ almost every $z$, that is, $\nu$ is $p$-mean $\alpha$-porous. The proof is finished by letting $\delta \downarrow 0$.
\end{proof}

\subsection{Annular porosity}
In this section, we prove Theorem \ref{thm:annularporosity} which shows the results for annular porosity corresponding to Theorems \ref{thm:mainporosity} and \ref{thm:meanporosity}. Investigating the proofs of Theorems \ref{thm:mainporosity} and \ref{thm:meanporosity}, we see that the geometric information obtained from the definition of porosity is only used to show the required properties of the sets $\cA_\eps$ and $\cU_\eps$. The corresponding sets in the annular porosity case are
\begin{equation*}
  \cA_\eps^\circ := \{ \nu \in \cM_1 : \nu(B(y,\alpha \roo |y|)) \leq \epsilon \text{ for some } y \in A(0,c,1) \}
\end{equation*}
and
\begin{equation*}
  \cU_\eps^\circ := \{ \nu \in \cM_1 : \nu(\overline{B}(y,\alpha \roo |y|)) < \epsilon \text{ for some } y \in A(0,c,1) \}.
\end{equation*}
To prove Theorem \ref{thm:annularporosity}, it in fact suffices to show that $\cA_\eps^\circ$ is closed and $\cU_\eps^\circ$ is open. Since the proof for the openness is the same as that of Lemma \ref{lem:opencollection} we just verify the closedness.

\begin{lemma} \label{lem:annularclosedcollection}
The set $\cA_\eps^\circ$ is closed for all $\eps\geq0$.
\end{lemma}

\begin{proof}
Suppose that there are a sequence $(\nu_i)_i$ of measures in $\AA_\eps^\circ$ and $\nu \in \cM_1$ so that $\nu_i \to \nu$ weakly. Fix $\alpha' < \alpha$. Choose $y_i \in A(0,c,1)$ with $\nu_i(B(y_i,\alpha \roo |y_i|)) \leq \epsilon$. By the compactness of $A(0,c,1)$, after possibly passing to a subsequence, we may assume that $y_i \to y \in A(0,c,1)$. Since $\alpha' < \alpha$, we have $B(y,\alpha' \roo |y|) \subset B(y_i,\alpha \roo |y_i|)$ for all large enough $i$. Thus by the openness of $B(y,\alpha \roo |y|)$, we have
$$\nu(B(y,\alpha' \roo |y|)) \leq \liminf_{i \to \infty} \nu_i(B(y,\alpha' \roo |y|)) \leq \liminf_{i \to \infty} \nu_i(B(y_i,\alpha \roo |y_i|)) \leq \epsilon.$$
Letting $\alpha' \uparrow \alpha$ gives $\nu(B(y,\alpha \roo|y|)) \leq \eps$. In particular, $\nu \in \cA_\epsilon^\circ$, showing that $\cA_\epsilon^\circ$ is closed.
\end{proof}

\section*{Acknowledgements}

We thank M. Hochman and V. Suomala for several discussions related to the topics of this paper, and the referee for useful comments and suggestions.

\bibliographystyle{abbrv} 
\bibliography{fractal_distributions-GMT_FINAL}

\end{document}